\documentclass{amsart}
\usepackage[backref]{hyperref}
\hypersetup{hidelinks}
\newtheorem{theorem}{Theorem}[section]
\newtheorem{lemma}{Lemma}
\newtheorem{assumption}{Assumption}
\newtheorem{corollary}{Corollary}

\theoremstyle{definition}

\theoremstyle{remark}
\newtheorem{remark}{Remark}

\numberwithin{equation}{section}

\usepackage{cleveref}
\usepackage{graphicx}
\usepackage{epstopdf}
\usepackage{cite}
\usepackage{color,soul}
\crefname{theorem}{Theorem}{Theorems}
\crefformat{equation}{(#2#1#3)}
\Crefformat{equation}{Eq.~(#2#1#3)}
\crefrangeformat{equation}{(#3#1#4)-(#5#2#6)}
\Crefrangeformat{equation}{Eqs.~(#3#1#4)-(#5#2#6)}
\crefmultiformat{equation}{(#2#1#3)}{ and~(#2#1#3)}{, (#2#1#3)}{ and~(#2#1#3)}
\Crefmultiformat{equation}{Eqs.~(#2#1#3)}{ and~(#2#1#3)}{, (#2#1#3)}{ and~(#2#1#3)}
\crefrangemultiformat{equation}{(#3#1#4)}{-(#5#2#6)}{, (#3#1#4)}{-(#5#2#6)}
\Crefrangemultiformat{equation}{Eqs.~(#3#1#4)}{-(#5#2#6)}{, (#3#1#4)}{-(#5#2#6)}
\crefname{lemma}{Lemma}{Lemmas}
\crefname{figure}{Fig.}{Figs.}
\crefname{example}{Example}{Examples}
\crefname{remark}{Remark}{Remarks}
\crefname{table}{Table}{Tables}
\crefname{assumption}{Assumption}{Assumptions}
\crefname{section}{Section}{Sections}

\begin{document}

\title[An EMA-conserving, pressure-robust and Re-semi-robust method]{An EMA-conserving, pressure-robust and Re-semi-robust reconstruction method for the unsteady incompressible Navier-Stokes equations}
\author{Xu Li}
\address{School of Mathematics, Shandong University, Jinan 250100, China}
\email{xulisdu@126.com}
\thanks{The second author is the corresponding author.}

\author{Hongxing Rui}
\address{School of Mathematics, Shandong University, Jinan 250100, China}
\email{hxrui@sdu.edu.cn}
\thanks{This work was supported by the National Natural Science Foundation of China grant 12131014.}

\subjclass[2020]{65M12, 65M15, 65M60, 76D05, 76D17}



\keywords{Finite element methods, unsteady Navier-Stokes equations, pressure-robustness, EMAC formulation, $Re$-semi-robustness}

\begin{abstract}
Proper EMA-balance (E: kinetic energy; M: momentum; A: angular momentum), pressure-robustness and $Re$-semi-robustness ($Re$: Reynolds number) are three important properties of Navier-Stokes simulations with exactly divergence-free elements. This EMA-balance makes a method conserve kinetic energy, linear momentum and angular momentum under some suitable senses; pressure-robustness means that the velocity errors are independent of the continuous pressure; $Re$-semi-robustness means that the constants appearing in the error bounds of kinetic and dissipation energies do not explicitly depend on inverse powers of the viscosity. In this paper, based on the pressure-robust reconstruction methods in [{A. Linke and C. Merdon, {\it Comput. Methods Appl. Mech.
Engrg.} 311 (2016), 304-326}], we propose a novel reconstruction method for a class of non-divergence-free simplicial elements which admits almost all the above properties. The only exception is the energy balance, where kinetic energy should be replaced by a properly redefined discrete energy. Some numerical comparisons with exactly divergence-free methods, pressure-robust reconstructions and the EMAC scheme are provided to confirm our theoretical results.
\end{abstract}

\maketitle

\section{Introduction}
\label{sec:introduction}
In this paper, we study the finite element methods for the unsteady Navier-Stokes equations (NSEs):
\begin{subequations}\label{NSE}
\begin{align}
\boldsymbol{u}_{t}-\nu\Delta\boldsymbol{u}+\left(\boldsymbol{u}\cdot\nabla\right)\boldsymbol{u}+\nabla p&=\boldsymbol{f}&\text{in} ~J\times\Omega,\label{NSE1}\\
\nabla\cdot\boldsymbol{u}&=0 &\text{in} ~J\times\Omega,\label{NSE2}\\
\boldsymbol{u}\left(0\right)&=\boldsymbol{u}^{0} &\text{in} ~\Omega,\label{NSE3}\\
\boldsymbol{u}&=\boldsymbol{0}&\text{on} ~J\times\Gamma,\label{NSE4}
\end{align}
\end{subequations}
where $J=\left(0,T\right]$ and $\Omega\subset \mathbb{R}^{d}$ $\left(d=2,3\right)$ is a bounded domain with Lipschitz-continuous polyhedral boundary $\Gamma$; $\boldsymbol{u}$ and $p$ represent the unknown velocity and pressure, respectively; $\nu>0$ is the constant kinematic viscosity; $\boldsymbol{f}$ represents the external force and $\boldsymbol{u}^{0}$ is the initial velocity. We assume $\boldsymbol{f}\left(t\right)\in \boldsymbol{L}^{2}\left(\Omega\right)$ for all $t\in J$. For simplicity, here we only consider the homogeneous Dirichlet boundary condition. Other boundary conditions are also of interest.

For designing accurate numerical schemes, it is widely believed that preserving the fundamental (physical or mathematical) properties of the continuous problem is of great importance. For the unsteady incompressible Navier-Stokes equations, these fundamental properties include the divergence constraint \cref{NSE2}, the balance laws for some physical quantities (e.g., kinetic energy, linear momentum, angular momentum, vorticity, helicity and enstrophy) \cite{Evans20132,Rebholz2017}, an invariance property for the velocity with respect to the gradient field in $\boldsymbol{f}$ \cite{Linke2014on,john_divergence_2017} and so on. Among these properties, the divergence constraint is of central importance. The papers \cite{Evans20132,john_divergence_2017} respectively showed that the exactly divergence-free mixed methods preserved the balance laws and the invariance property mentioned above. The latter means that these methods are pressure-robust; namely the velocity errors are independent of the pressure. Moreover, it was demonstrated in \cite{schroeder_towards_2018} that the constants in error estimates, including the Gronwall constant, did not depend on $\nu^{-1}$ explicitly for divergence-free finite element methods. This property was called ($Re$-)semi-robustness, uniform, or quasi-uniform estimates \cite{john_finite_2018} sometimes.

Due to these fascinating properties, constructing exactly divergence-free elements has been an increasingly hot topic in recent years \cite{Zhang2005,guzman_conforming_2013,Neilan2014,Neilan2015,Neilan2018,christiansen_generalized_2018,Neilan2021}. However, compared to classical non-divergence-free elements (e.g., Taylor-Hood, MINI and Bernardi-Raugel, cf. \cite{girault_finite_1986,John2016}), the construction of these elements is not trivial in most cases.
Another popular idea is relaxing the continuity condition but enforcing the divergence constraint strongly, which results in the so-called nonconforming $H(\operatorname{div})$-conforming methods \cite{Cockburn2007,wang_new_2007,wang2008,konno2011,johnny_family_2012,Rhebergen2020}. In this paper, we focus on the conforming mixed methods for the Navier-Stokes equations.

Modifying the formulation to preserve some fundamental properties of the continuous problem (or divergence-free mixed methods) for non-divergence-free elements is another popular research topic. With an observation that most classical elements are non-divergence-free, {this topic includes} (not just) pressure-robust reconstructions \cite{Linke2014on,linke_robust_2016,linke2016,Lederer2016,Linke2017}, the EMAC (EMA-conserving) formulation \cite{Rebholz2017,EMAC2019,Rebholz2020} and some $Re$-semi-robust methods \cite{graddiv2018,CIP2007,LPS2019,SPS2021}.
The method introduced in this paper is also included in this topic. Pressure-robustness plays an important role on the accuracy of a method when {\it `gradient forces dominate the momentum balance'} \cite{linke2019jcp}; the velocity errors of the methods which are not $Re$-semi-robust might grow quickly with respect to time for higher Reynolds number flows \cite{schroeder_towards_2018}; the EMAC scheme is one of the ``enhanced-physics" based schemes which have a long history such as \cite{Arakawa1966,Fix1975,Abramov2003,Rebholz2007,Palha2017}, and the paper \cite{Rebholz2020} showed that an improper treatment of energy, linear momentum and angular momentum produced lower bounds for $L^{2}$ velocity errors. It is worth mentioning that, the properties mentioned above are usually not mutually independent. For example, in the paper \cite{Rebholz2020}, Olshanskii and Rebholz proved that the Gronwall constants in EMAC estimates did not depend on the viscosity explicitly, which is exactly $Re$-semi-robustness except that the constant in the pressure-induced error polynomially depended on the inverse of viscosity. Another example is the popular grad-div stabilization \cite{Olshanskii2004,Linke2011}. It can not only weaken (not totally remove) the impact of the pressure on velocity errors \cite{john_divergence_2017} but also make the usual skew-symmetric scheme semi-robust with respect to the Reynolds number \cite{graddiv2018}. Finally, we also refer the readers to the review article \cite{john_finite_2018} for more details.

In this paper, we propose a novel reconstruction formulation which is EMA-conserving for the reconstructed discrete velocity (here the energy should be redefined), pressure-robust and $Re$-semi-robust. For simplicity, we shall refer to this reconstruction as the ``EMAPR" reconstruction throughout this paper. Our method is based on the pressure-robust reconstruction formulation in \cite{linke2016}. The main difference lies on the discretization of the convective term. For the convective term, two (pressure-robust) discrete forms were proposed in \cite{linke2016}: the convective form and the rotational form. Similarly to \cite{Rebholz2017}, it can be checked that the two forms do not conserve the linear momentum and angular momentum (the latter conserves kinetic energy). Here we propose an EMA-conserving form, i.e., it does not produce any extra energy, momentum and angular momentum under some appropriate senses. Then we give a pressure-robust and $Re$-semi-robust error estimate for the continuous-in-time case, provided that the continuous solution $\boldsymbol{u}$ is in $L^{2}\left(J;\boldsymbol{W}^{1,\infty}\left(\Omega\right)\right)$. To obtain such an estimate, we also need to slightly modify the discretization of the evolutionary term by introducing a stabilization. Finally, we shall prove that our formulation could be easily applied to a class of simplicial locally mass-conserving elements whose pressures are discontinuous. To the best of our knowledge, the EMAPR reconstruction is the first method on conforming non-divergence-free elements which is EMA-conserving, pressure-robust and $Re$-semi-robust simultaneously. For nonconforming and non-divergence-free methods, a reconstructed Hybrid discontinuous Galerkin method in \cite{lederer_hybrid_2019} (see formulas ``(6.3d)" and ``(6.5)" in it) probably admits most of these properties also.

The remainder of this paper is organized as follows. In \cref{sec:2} we discuss the EMAPR methods and some balance laws. \Cref{sec:3} is devoted to giving a pressure-robust and $Re$-semi-robust error estimate for the EMAPR method. We show that a class of locally divergence-free elements (include the Bernardi-Raugel element) could be easily incorporated into our framework in \cref{sec:4}. Finally we carry out some numerical experiments in \cref{sec:5}.

In what follows we will use $C$, with or without a subscript, to denote a generic positive constant. The standard inner product for $\left[L^{2}\left(D\right)\right]^{n}$ or $\left[L^{2}\left(D\right)\right]^{n\times n}$ ($n\in\mathbb{Z}$) will be denoted by $\left(\cdot,\cdot\right)_{D}$ uniformly. The notation $\left\|\cdot\right\|_{m,p,D}$ ($\left|\cdot\right|_{m,p,D}$) will be used to denote the Sobolev norm (seminorm, respectively) of $\left[W^{m,p}\left(D\right)\right]^{n}$ or $\left[W^{m,p}\left(D\right)\right]^{n\times n}$. With the convention the subscripts $m,p$ and $D$ will be omitted for $m=0$, $p=2$ and $D=\Omega$, respectively. $H^{m}\left(D\right)$ coincides with $W^{m,2}\left(D\right)$ and $\boldsymbol{W}^{m,p}\left(D\right)$ coincides with $\left[{W}^{m,p}\left(D\right)\right]^{d}$.
\section{The EMAPR reconstruction method}
\label{sec:2}
\subsection{The divergence-free reconstruction operator}
\label{sec:21}
Let $\mathcal{T}_{h}$ denote a partition of $\Omega$. We define the mesh size $h:=\max_{K\in\mathcal{T}_{h}}h_{K}$ with $h_{K}$ the diameter of elements $K$. Denote by $\varrho_{K}$ the diameter of the biggest ball inscribed in $K$. Here we assume that $\mathcal{T}_{h}$ is shape-regular \cite{Ciarlet2002The}, i.e., there exists a positive constant $\xi$ such that
\begin{equation}\label{shaperegularity}
\frac{h_{K}}{\varrho_{K}}\leq \xi \quad \forall~K\in\mathcal{T}_{h}.
\end{equation}
Introduce
$$V=\boldsymbol{H}^{1}_{0}\left(\Omega\right):=\left\{\boldsymbol{v}\in \boldsymbol{H}^{1}\left(\Omega\right): \boldsymbol{v}|_{\Gamma}=\boldsymbol{0}\right\},$$
$$X=\boldsymbol{H}_{0}\left(\operatorname{div};\Omega\right):=\left\{\boldsymbol{v}\in \boldsymbol{H}\left(\operatorname{div};\Omega\right): \boldsymbol{v}\cdot\boldsymbol{n}\right|_{\Gamma}=0\},$$
$$\boldsymbol{H}^{2}\left(\mathcal{T}_{h}\right):=\left\{\boldsymbol{v}\in \boldsymbol{L}^{2}\left(\Omega\right): \boldsymbol{v}|_{K}\in \boldsymbol{H}^{2}\left(K\right) \quad \forall~ K\in\mathcal{T}_{h}\right\},$$
and
$$W=L_{0}^{2}\left(\Omega\right):=\left\{q\in L^{2}\left(\Omega\right): \int_{\Omega}q~d\boldsymbol{x}=0\right\},$$
where $\boldsymbol{n}$ is the unit external normal vector on $\Gamma$.
Furthermore we define the bilinear form $b: X\times W\rightarrow \mathbb{R}$ by
\begin{displaymath}
b\left(\boldsymbol{v},q\right):=\left(\nabla\cdot\boldsymbol{v},q\right)\ \text{ for any } \ \left(\boldsymbol{v},q\right)\in X\times W.
\end{displaymath}
Let $\left(V_{h},X_{h},W_{h}\right)\subset \left(V,X,W\right)$ denotes a triple of finite element spaces satisfying
\begin{equation}\label{infsupcondition}
\inf_{q_{h}\in W_{h}}\sup_{\boldsymbol{v}_{h}\in V_{h}\setminus\left\{\boldsymbol{0}\right\}}\frac{b\left(\boldsymbol{v}_{h},q_{h}\right)}{\left\|\nabla\boldsymbol{v}\right\|}\geq \beta_{h}\left\|q_{h}\right\| \text { for some } \beta_{h}>0,
\end{equation}
and
\begin{equation}\label{divfreeinclude}
\nabla\cdot X_{h}\subseteq W_{h}.
\end{equation}
\begin{remark}
For simplicity, throughout this article we assume that  $\boldsymbol{u}\notin V_{h}$ for the true solution $\boldsymbol{u}$. This assumption does not influence the construction of the method, but will be beneficial to simplifying notation in error estimates and highlighting the fundamental ideas. For a more general velocity, there is no extra essential difficulty for analysis.
\end{remark}

Denote by
\begin{displaymath}
V^{0}:=\left\{\boldsymbol{v}\in V: b\left(\boldsymbol{v},q\right)=0 \quad\forall ~q\in W\right\}=\left\{\boldsymbol{v}\in V: \nabla\cdot\boldsymbol{v}=0\right\},
\end{displaymath}
and
\begin{displaymath}
V_{h}^{0}:=\left\{\boldsymbol{v}_{h}\in V_{h}: b\left(\boldsymbol{v}_{h},q_{h}\right)=0 \quad\forall ~q_{h}\in W_{h}\right\},
\end{displaymath}
the spaces of divergence-free velocity functions and discretely divergence-free velocity functions, respectively.
Note that if $\nabla\cdot V_{h} \subseteq W_{h}$, we have $V_{h}^{0}\subset V^{0}$, which means that the functions in $V_{h}^{0}$ are exactly divergence-free. For most classical elements, this relationship does not hold.

Let $P^{k}\left(K\right)$ denotes the space of polynomials on $K$ of degree no more than $k$. We also define
\begin{displaymath}
\boldsymbol{P}^{k}:=\left\{\boldsymbol{v}_{h}\in V: \boldsymbol{v}_{h}|_{K}\in \left[P^{k}(K)\right]^{d}\text{ for all } K\in\mathcal{T}_{h}\right\}.
\end{displaymath}
We suppose that the velocity space $V_{h}$ is of order $k$ ($k\geq 1$), i.e., there exists a non-negative integer $k$ such that $\boldsymbol{P}^{k}\subset V_{h}$ and $\boldsymbol{P}^{k+1}\not\subset V_{h}$.

We introduce the divergence-free reconstruction operator $\Pi_{h}: V_{h}\rightarrow X_{h}$ (we do not give the concrete definition here) which satisfies that
\begin{equation}\label{divfreedefinition}\nabla\cdot\Pi_{h}\boldsymbol{v}_{h}\equiv 0  \text{ for all } \boldsymbol{v}_{h}\in V_{h}^{0};\end{equation}
\begin{equation}\label{bconsistency}b\left(\boldsymbol{v}_{h},q_{h}\right)=b\left(\Pi_{h}\boldsymbol{v}_{h},q_{h}\right) \text{ for all } \left(\boldsymbol{v}_{h},q_{h}\right)\in V_{h}\times W_{h};\end{equation}
\begin{equation}\label{consistency}
\left(\boldsymbol{g}, \boldsymbol{v}_{h}-\Pi_{h} \boldsymbol{v}_{h}\right) \leq C h^{k}\left|\boldsymbol{g}\right|_{k-1}\left\|\nabla \boldsymbol{v}_{h}\right\|  \text { for all } \boldsymbol{g} \in \boldsymbol{H}^{k-1}\left(\Omega\right), \boldsymbol{v}_{h}\in V_{h}^{0}.
\end{equation}
Note that \cref{divfreedefinition} can be derived from \cref{divfreeinclude} and \cref{bconsistency}.

We also assume that $\Pi_{h}$ satisfies the following properties.
\begin{assumption}\label{assumption1}
There exists two operators $\Pi_{h}^{1}: V_{h}\rightarrow V_{h}$ and $\Pi_{h}^{R}: V_{h}\rightarrow X_{h}$ such that $\Pi_{h}\boldsymbol{v}=\Pi_{h}^{1}\boldsymbol{v}+\Pi_{h}^{R}\boldsymbol{v}$ for all $\boldsymbol{v}\in V_{h}$ and
\begin{equation}\label{boundassumption}
\left|\Pi_{h}^{1}\boldsymbol{v}_{h}\right|_{1,\infty,K}\leq C \left|\boldsymbol{v}_{h}\right|_{1,\infty,K}\ \ \forall~ K\in \mathcal{T}_{h}, \boldsymbol{v}_{h}\in V_{h}^{0},
\end{equation}
\begin{equation}\label{convergenceassumption}
\left\|\Pi_{h}^{R}\boldsymbol{v}_{h}\right\|_{\infty,K}\leq C h_{K} \left|\boldsymbol{v}_{h}\right|_{1,\infty,K}\ \ \forall~ K\in \mathcal{T}_{h}, \boldsymbol{v}_{h}\in V_{h}^{0}.
\end{equation}
\end{assumption}
Finally, we extend the definitions of $\Pi_{h}$, $\Pi_{h}^{1}$ and $\Pi_{h}^{R}$ to $\operatorname{span}\left\{{\boldsymbol{u}}\right\}$ for the exact solution $\boldsymbol{u}$, by defining that
\begin{equation}\label{pihontruesolution}
\Pi_{h}\boldsymbol{u}=\Pi_{h}^{1}\boldsymbol{u}=\boldsymbol{u},\quad \Pi_{h}^{R}\boldsymbol{u}=\boldsymbol{0}.
\end{equation}
Since we have assumed $\boldsymbol{u}\notin V_{h}$, this extension will not arise contradiction with their definitions on $V_{h}$.

In \cref{sec:4}, we shall show that the reconstruction operators in \cite[Remark 4.2]{linke2016} and their higher order versions on a class of locally divergence-free simplicial elements satisfy all the above properties.
\subsection{The EMAPR method for classical elements}
\label{sec:22}
Introduce
\begin{displaymath}
a\left(\boldsymbol{u}, \boldsymbol{v}\right):=\left(\nabla\boldsymbol{u},\nabla\boldsymbol{v}\right)\text{ for all } \boldsymbol{u}, \boldsymbol{v}\in V,
\end{displaymath}
and
\begin{displaymath}
c\left(\boldsymbol{u},\boldsymbol{v},\boldsymbol{w}\right):=\left(\left(\boldsymbol{u}\cdot\nabla\right)\boldsymbol{v},\boldsymbol{w}\right) \text{ for all } \left(\boldsymbol{u},\boldsymbol{v},\boldsymbol{w}\right)\in
\boldsymbol{L}^{2}\left(\Omega\right)\times V\times \boldsymbol{L}^{2}\left(\Omega\right).
\end{displaymath}

The weak formulation for \cref{NSE} characterises $\left(\boldsymbol{u},p\right): J \rightarrow V\times W$ by
\begin{subequations}\label{weakcontinuousform}
\begin{align}
\left(\frac{\partial\boldsymbol{u}}{\partial t},\boldsymbol{v}\right)+\nu a\left(\boldsymbol{u}, \boldsymbol{v}\right)+c\left(\boldsymbol{u},\boldsymbol{u},\boldsymbol{v}\right)-b\left(\boldsymbol{v}, p\right) &=\left(\boldsymbol{f}, \boldsymbol{v}\right)
\quad\forall~\boldsymbol{v}\in V,\label{weakcontinuousform1} \\
b\left(\boldsymbol{u}, q\right) &=0\quad\  \forall~q\in W,\label{weakcontinuousform2}
\end{align}
\end{subequations}
and $\boldsymbol{u}\left(0\right)=\boldsymbol{u}^{0}$. A straightforward semi-discrete analog of \cref{weakcontinuousform} is to find $\left(\boldsymbol{u}_{h},p_{h}\right): J \rightarrow V_{h}\times W_{h}$ satisfying
$\boldsymbol{u}_{h}\left(0\right)=\boldsymbol{u}_{h}^{0}\in V_{h}$ with $\boldsymbol{u}_{h}^{0}$ some approximation of $\boldsymbol{u}^{0}$ and
\begin{subequations}\label{discreteanalog}
\begin{align}
\left(\frac{\partial\boldsymbol{u}_{h}}{\partial t},\boldsymbol{v}_{h}\right)+\nu a\left(\boldsymbol{u}_{h}, \boldsymbol{v}_{h}\right)+c\left(\boldsymbol{u}_{h},\boldsymbol{u}_{h},\boldsymbol{v}_{h}\right)-b\left(\boldsymbol{v}_{h}, p_{h}\right) &=\left(\boldsymbol{f}, \boldsymbol{v}_{h}\right),\label{discreteanalog1} \\
b\left(\boldsymbol{u}_{h}, q_{h}\right) &=0,\label{discreteanalog2}
\end{align}
\end{subequations}
for all $\left(\boldsymbol{v}_{h},q_{h}\right)\in V_{h}\times W_{h}$.
However, it is well-known that the above scheme is not energy-stable and pressure-robust unless $\boldsymbol{u}_{h}$ is exactly divergence-free (or equivalently, $\nabla\cdot V_{h}\subseteq W_{h}$). To obtain a pressure-robust velocity for classical elements, in \cite{linke2016} Linke and Merdon proposed a novel finite element formulation which reads
\begin{displaymath}
\begin{aligned}
\left(\Pi_{h}\frac{\partial\boldsymbol{u}_{h}}{\partial t},\Pi_{h}\boldsymbol{v}_{h}\right)&+\nu a\left(\boldsymbol{u}_{h}, \boldsymbol{v}_{h}\right)+c\left(\boldsymbol{u}_{h},\boldsymbol{u}_{h},\Pi_{h}\boldsymbol{v}_{h}\right)\\&-b\left(\boldsymbol{v}_{h}, p_{h}\right) +
b\left(\boldsymbol{u}_{h}, q_{h}\right) =\left(\boldsymbol{f}, \Pi_{h}\boldsymbol{v}_{h}\right).
\end{aligned}
\end{displaymath}
Here $b\left(\boldsymbol{v}_{h},p_{h}\right)$ should be interpreted as $b\left(\Pi_{h}\boldsymbol{v}_{h},p_{h}\right)$ via property \cref{bconsistency}. By using divergence-free reconstructions, the above formulation restores the $L^{2}$-orthogonality between discretely divergence-free test functions and gradient fields, and thus remove the effect of the continuous pressure for velocity errors. There is a consistency error arising from the diffusion term.

To make the method energy-stable, a pressure-robust and energy-conserving discretization (the rotational form) of the nonlinear term were also proposed in \cite{linke2016}:
\begin{equation}\label{rotationalform}
c_{\operatorname{rot}}\left(\boldsymbol{u}_{h},\boldsymbol{v}_{h},\boldsymbol{w}_{h}\right):=\left(\nabla\times\boldsymbol{u}_{h}\times\Pi_{h}\boldsymbol{v}_{h},\Pi_{h}\boldsymbol{w}_{h}\right),
\end{equation}
where $\nabla\times$ is the $\operatorname{curl}$ operator \cite{girault_finite_1986}.
However, following \cite{Rebholz2017} or \cref{sec:23} below, one can prove that the rotational form does not preserve momentum and angular momentum (see \cref{remark:momentumforRot} below). To resolve this issue, we propose an EMA-conserving form, which results in the EMAPR reconstruction:
\begin{subequations}\label{EMAPRsemidiscreteform}
\begin{align}
{\rm Find}~\left(\boldsymbol{u}_{h},p_{h}\right): J\rightarrow V_{h}\times W_{h}~ {\rm such ~that}\quad\quad\quad\quad\quad\qquad&\nonumber\\
d_{h}\left(\frac{\partial\boldsymbol{u}_{h}}{\partial t},\boldsymbol{v}_{h}\right)+\nu a\left(\boldsymbol{u}_{h}, \boldsymbol{v}_{h}\right)+c_{h}\left(\boldsymbol{u}_{h},\boldsymbol{u}_{h},\boldsymbol{v}_{h}\right)-b\left(\boldsymbol{v}_{h}, p_{h}\right) &=\left(\boldsymbol{f}, \Pi_{h}\boldsymbol{v}_{h}\right), \label{EMAPRsemidiscreteform1} \\
b\left(\boldsymbol{u}_{h}, q_{h}\right) &=0,\qquad\ \ \quad\label{EMAPRsemidiscreteform2}
\end{align}
\end{subequations}
for all $\boldsymbol{v}_{h}\in V_{h},q_{h}\in W_{h}$ and $\boldsymbol{u}_{h}\left(0\right)=\boldsymbol{u}_{h}^{0}$.
Here $d_{h}$ is given by
\begin{equation}\label{dtbilinearform}
d_{h}\left(\boldsymbol{u}_{h},\boldsymbol{v}_{h}\right):=\left(\Pi_{h}\boldsymbol{u}_{h},\Pi_{h}\boldsymbol{v}_{h}\right)+\alpha\left(\Pi_{h}^{R}\boldsymbol{u}_{h},\Pi_{h}^{R}\boldsymbol{v}_{h}\right),
\end{equation}
where $\alpha$ is a positive parameter.
The trilinear form $c_{h}$ is defined by
\begin{equation}\label{EMAPRtrilinearform}
\begin{aligned}
c_{h}\left(\boldsymbol{u}_{h},\boldsymbol{v}_{h},\boldsymbol{w}_{h}\right):&=c\left(\Pi_{h}\boldsymbol{u}_{h},\Pi_{h}^{1}\boldsymbol{v}_{h},\Pi_{h}^{1}\boldsymbol{w}_{h}\right) +c\left(\Pi_{h}\boldsymbol{u}_{h},\Pi_{h}^{1}\boldsymbol{v}_{h},\Pi_{h}^{R}\boldsymbol{w}_{h}\right)\\&\ \ \ \ \ -
c\left(\Pi_{h}\boldsymbol{u}_{h},\Pi_{h}^{1}\boldsymbol{w}_{h},\Pi_{h}^{R}\boldsymbol{v}_{h}\right)\\
&=c\left(\Pi_{h}\boldsymbol{u}_{h},\Pi_{h}^{1}\boldsymbol{v}_{h},\Pi_{h}\boldsymbol{w}_{h}\right)-
c\left(\Pi_{h}\boldsymbol{u}_{h},\Pi_{h}^{1}\boldsymbol{w}_{h},\Pi_{h}^{R}\boldsymbol{v}_{h}\right).
\end{aligned}
\end{equation}
\begin{remark}
The fundamental requirement of projection $\Pi_{h}$ is: For any $\boldsymbol{v}_{h}\in V_{h}$, $\Pi_{h}\boldsymbol{v}_{h}$ can be decomposed into a sufficiently approximate $H^{1}$-conforming component and a
``small" $H(\operatorname{div})$-conforming component (consider $\Pi_{h}^{1}\boldsymbol{v}_{h}$ and $\Pi_{h}^{R}\boldsymbol{v}_{h}$). This is the prerequisite of constructing our methods. Regarding the discretization of the convection term, a rough description of our basic idea is: Apply the $H^{1}$-conforming part to guarantee the accuracy and ``abuse" the $H(\operatorname{div})$-conforming part to guarantee the conservation of energy, momentum and angular momentum. We note that the reconstruction operators in \cite{Linke2017} for Taylor-Hood and MINI elements probably satisfy this fundamental requirement also.
\end{remark}
\begin{remark}
The bilinear form $d_{h}$ with $\alpha=0$ is the classical discretization in pressure-robust reconstructions for the $\left(\boldsymbol{u},\boldsymbol{v}\right)$-like term. However, for the case $\alpha=0$ we can not obtain a $Re$-semi-robust estimate theoretically. In practice, we find that the stabilization term $\left(\Pi_{h}^{R}\boldsymbol{u}_{h},\Pi_{h}^{R}\boldsymbol{v}_{h}\right)$ is of importance for the high order locally divergence-free elements ($k\geq 2$) in the case that $\nu$ is very small or equal to zero (the Euler equation). In this case, without this term, the $H^{1}$ error of the discrete velocity might be large.
\end{remark}
\subsection{EMA-balance in semi-discrete schemes}
\label{sec:23}
Now we are in the position to analyze the discrete balance laws with the EMAPR reconstruction. We define kinetic energy $E: X\rightarrow \mathbb{R}$, momentum $M: X\rightarrow \mathbb{R}^{d}$ and angular momentum $M_{\boldsymbol{x}}: X\rightarrow \mathbb{R}^{3}$ by
$$E\left(\boldsymbol{u}^{*}\right):=\frac{1}{2} \int_{\Omega}|\boldsymbol{u}^{*}|^{2} {~d} \boldsymbol{x},\quad
M\left(\boldsymbol{u}^{*}\right):=\int_{\Omega} \boldsymbol{u}^{*} ~{d} \boldsymbol{x},\quad
M_{\boldsymbol{x}}\left(\boldsymbol{u}^{*}\right):=\int_{\Omega} \boldsymbol{u}^{*} \times \boldsymbol{x} ~{d} \boldsymbol{x},$$
for any $\boldsymbol{u}^{*}\in X$. For any two-dimensional vector $\boldsymbol{u}^{*}=\left(u_{1}^{*},u_{2}^{*}\right)$, to compute angular momentum or cross product, one can always embed it into three-dimensional spaces by setting $\boldsymbol{u}^{*}=\left(u_{1}^{*},u_{2}^{*},0\right)$.
Let $\boldsymbol{u}$ be the solution of \cref{weakcontinuousform} and it satisfies the following balance laws \cite{Rebholz2017,Rebholz2020}:
\begin{displaymath}
\frac{d}{dt}E\left(\boldsymbol{u}\right)+\nu\left\|\nabla\boldsymbol{u}\right\|^{2}=\left(\boldsymbol{f},\boldsymbol{u}\right),\quad \frac{d}{dt}M\left(\boldsymbol{u}\right)=\int_{\Omega}\boldsymbol{f}~d\boldsymbol{x}, \quad \frac{d}{dt}M_{\boldsymbol{x}}\left(\boldsymbol{u}\right)=\int_{\Omega}\boldsymbol{f}\times \boldsymbol{x}~d\boldsymbol{x},
\end{displaymath}
where the balance laws of momentum and angular momentum are based on some appropriate assumptions. Following \cite{Rebholz2017,EMAC2019,Rebholz2020}, here (only for the analysis of momentum and angular momentum) we assume that $\left(\boldsymbol{u},p\right)$ is compactly supported in $\Omega$ (e.g., consider an isolated vortex). We also define a discrete energy $E_{d}: V_{h}^{0}\oplus\operatorname{span}\left\{\boldsymbol{u}\right\}\rightarrow \mathbb{R}$ by
 $$E_{d}\left(\boldsymbol{u}^{*}\right):=\frac{1}{2} d_{h}\left(\boldsymbol{u}^{*},\boldsymbol{u}^{*}\right),$$
for any $\boldsymbol{u}^{*}\in V_{h}^{0}\oplus\operatorname{span}\left\{\boldsymbol{u}\right\}$. Note that we have $E\left(\Pi_{h}\boldsymbol{u}^{*}\right)\leq  E_{d}\left(\boldsymbol{u}^{*}\right)$ with $E\left(\boldsymbol{u}\right)=E_{d}\left(\boldsymbol{u}\right)$.

The following lemma is essential for EMA analysis.
\begin{lemma}\label{essentiallemma}
For any finite element triple $\left(\boldsymbol{u}_{h},\boldsymbol{v}_{h},\boldsymbol{w}_{h}\right)\in {\boldsymbol{H}}\left(\operatorname{div};\Omega\right)\times{\boldsymbol{H}^{1}}\left(\Omega\right)\times{\boldsymbol{H}^{1}}\left(\Omega\right)$, we have
\begin{equation}\label{essentialformu}
c\left(\boldsymbol{u}_{h},\boldsymbol{v}_{h},\boldsymbol{w}_{h}\right)=-c\left(\boldsymbol{u}_{h},\boldsymbol{w}_{h},\boldsymbol{v}_{h}\right)
\end{equation} if
\begin{itemize}
   \item[\text{1)}]$\boldsymbol{u}_{h}$ is exactly divergence-free, i.e., $\nabla\cdot\boldsymbol{u}_{h}\equiv0$;
   \item[\text{2)}] $\boldsymbol{u}_{h}|_{\Gamma}\cdot\boldsymbol{n}=0$ or $\boldsymbol{v}_{h}|_{\Gamma}=\boldsymbol{0}$ or $\boldsymbol{w}_{h}|_{\Gamma}=\boldsymbol{0}$.
\end{itemize}
\end{lemma}
\begin{proof}
This lemma is covered by the lemma for skew-symmetry of a class of discontinuous Galerkin formulations, cf. \cite[Lemma 6.39]{di_pietro_mathematical_2012}. In fact, the trilinear form $``t_{h}\left(\cdot,\cdot,\cdot\right)"$ in \cite[Lemma 6.39]{di_pietro_mathematical_2012} is exactly $c\left(\cdot,\cdot,\cdot\right)$ provided that all inputs satisfy the conditions in \cref{essentiallemma}. Then \cite[Lemma 6.39]{di_pietro_mathematical_2012} implies, for $\left(\boldsymbol{u}_{h},\boldsymbol{v}_{h},\boldsymbol{w}_{h}\right)$ satisfying the conditions in \cref{essentiallemma},
\begin{subequations}\nonumber
\begin{align}
0=c\left(\boldsymbol{u}_{h},\boldsymbol{v}_{h}+\boldsymbol{w}_{h},\boldsymbol{v}_{h}+\boldsymbol{w}_{h}\right)=&\underbrace{c\left(\boldsymbol{u}_{h},\boldsymbol{v}_{h},\boldsymbol{v}_{h}\right)}_{=0}
+\underbrace{c\left(\boldsymbol{u}_{h},\boldsymbol{w}_{h},\boldsymbol{w}_{h}\right)}_{=0}+\\ &c\left(\boldsymbol{u}_{h},\boldsymbol{v}_{h},\boldsymbol{w}_{h}\right)+c\left(\boldsymbol{u}_{h},\boldsymbol{w}_{h},\boldsymbol{v}_{h}\right).
\end{align}
\end{subequations}
This completes the proof.
\end{proof}
\cref{essentiallemma} and \cref{EMAPRtrilinearform} imply the following lemma.
\begin{lemma}\label{secondessentiallemma}
For any $\left(\boldsymbol{u}_{h},\boldsymbol{v}_{h},\boldsymbol{w}_{h}\right)\in V_{h}^{0}\times V_{h}\times V_{h}$ we have
\begin{displaymath}
c_{h}\left(\boldsymbol{u}_{h},\boldsymbol{v}_{h},\boldsymbol{w}_{h}\right)=-c_{h}\left(\boldsymbol{u}_{h},\boldsymbol{w}_{h},\boldsymbol{v}_{h}\right).
\end{displaymath}
\end{lemma}
Setting $\boldsymbol{v}_{h}=\boldsymbol{u}_{h}$ in \cref{EMAPRsemidiscreteform} and applying \cref{EMAPRsemidiscreteform2} and \cref{secondessentiallemma},
 one immediately obtains that
\begin{theorem}\label{energytheorem}
Let $\boldsymbol{u}_{h}$ be the solution of \cref{EMAPRsemidiscreteform}. Then it satisfies the following balance of energy:
\begin{displaymath}
\frac{d}{dt}E_{d}\left(\boldsymbol{u}_{h}\right)+\nu\left\|\nabla\boldsymbol{u}_{h}\right\|^{2}=\left(\boldsymbol{f},\Pi_{h}\boldsymbol{u}_{h}\right).
\end{displaymath}
\end{theorem}
Boundary conditions may influence the balance of momentum and angular momentum. For simplicity, we do some extra assumptions which are similar to the continuous case, to remove the contribution of boundary. Similar assumptions were also used for the analysis of the EMAC formulation in \cite{Rebholz2017}.
\begin{assumption}\label{assumption2}
The finite element solution $\left(\boldsymbol{u}_{h},p_{h}\right)$, $\Pi_{h}\boldsymbol{u}_{h}$ and the external force $\boldsymbol{f}$ are only supported on a sub-mesh $\hat{\mathcal{T}_{h}}\subset\mathcal{T}_{h}$ such that there exists an operator $\chi: \boldsymbol{H}^{1}\left(\Omega\right)\rightarrow \boldsymbol{P}^{k}\subseteq V_{h}$ satisfying
\begin{equation}\label{assumption2formu}\chi\left(\boldsymbol{g}\right)|_{\hat{\mathcal{T}_{h}}}=\boldsymbol{g},\quad\Pi_{h}\chi\left(\boldsymbol{g}\right)|_{\hat{\mathcal{T}_{h}}}=\Pi_{h}^{1}\chi\left(\boldsymbol{g}\right)|_{\hat{\mathcal{T}_{h}}}=\boldsymbol{g},\end{equation}
for $\boldsymbol{g}=\boldsymbol{e}_{i},\boldsymbol{x}\times\boldsymbol{e}_{i}$ $\left(1\leq i\leq d\right)$. Here $\boldsymbol{e}_{i}\in\mathbb{R}^{d}$ is the unit vector whose $i$-th component is equal to 1.
\end{assumption}

In fact, for locally divergence-free elements, the support of $\Pi_{h}\boldsymbol{u}_{h}$ is the same as $\boldsymbol{u}_{h}$, since the reconstruction could be locally performed on each element. Furthermore, note that $\nabla\cdot\boldsymbol{e}_{i}=\nabla\cdot\left(\boldsymbol{x}\times\boldsymbol{e}_{i}\right)=0$, and $\boldsymbol{e}_{i}$ and $\boldsymbol{x}\times\boldsymbol{e}_{i}$ are linear polynomials. So the equalities for $\Pi_{h}$ and $\Pi_{h}^{1}$ in \cref{assumption2formu} are not hard to satisfy. The reconstruction operators in \cref{sec:4} fulfill these equalities.

Next, the following equalities will be used to analyze angular momentum:
\begin{equation}\label{equality1}
\boldsymbol{a}\cdot\left(\boldsymbol{b}\times\boldsymbol{c}\right)=\left(\boldsymbol{a}\times\boldsymbol{b}\right)\cdot\boldsymbol{c},
\end{equation}
and
\begin{equation}\label{equality2}
\left(\boldsymbol{a}\cdot\nabla\right)\left(\boldsymbol{x}\times\boldsymbol{e}_{i}\right)\cdot\boldsymbol{b}=-\left(\boldsymbol{a}\times\boldsymbol{b}\right)\cdot\boldsymbol{e}_{i},
\end{equation}
for any $\boldsymbol{a},\boldsymbol{b},\boldsymbol{c}\in\mathbb{R}^{3}$. These two equalities can be obtained by expanding out each term.

\begin{theorem}\label{MAconservingthm}
Let $\boldsymbol{u}_{h}$ be the solution of \cref{EMAPRsemidiscreteform}. Then under \cref{assumption2} $\Pi_{h}{u}_{h}$ satisfies the following balances:
\begin{displaymath}
\frac{d}{dt}M\left(\Pi_{h}\boldsymbol{u}_{h}\right)=\int_{\Omega}\boldsymbol{f}~d\boldsymbol{x}, \quad \frac{d}{dt}M_{\boldsymbol{x}}\left(\Pi_{h}\boldsymbol{u}_{h}\right)=\int_{\Omega}\boldsymbol{f}\times \boldsymbol{x}~d\boldsymbol{x}.
\end{displaymath}
\end{theorem}
\begin{proof}
Taking $\boldsymbol{v}_{h}=\chi\left(\boldsymbol{e}_{i}\right), \chi\left(\tilde{\boldsymbol{e}}_{i}\right)$ with $\tilde{\boldsymbol{e}}_{i}:=\boldsymbol{x}\times\boldsymbol{e}_{i}$ in \cref{EMAPRsemidiscreteform}, by \cref{assumption2} one respectively obtains
\begin{displaymath}
\left(\frac{\partial\left(\Pi_{h}\boldsymbol{u}_{h}\right)}{\partial t},\boldsymbol{e}_{i}\right)+\nu a\left(\boldsymbol{u}_{h}, \boldsymbol{e}_{i}\right)+c_{h}\left(\boldsymbol{u}_{h},\boldsymbol{u}_{h},\boldsymbol{e}_{i}\right)-b\left(\boldsymbol{e}_{i}, p_{h}\right) =\left(\boldsymbol{f}, \boldsymbol{e}_{i}\right),
\end{displaymath}
and
\begin{displaymath}
\left(\frac{\partial\left(\Pi_{h}\boldsymbol{u}_{h}\right)}{\partial t},\tilde{\boldsymbol{e}}_{i}\right)+\nu a\left(\boldsymbol{u}_{h}, \tilde{\boldsymbol{e}}_{i}\right)+c_{h}\left(\boldsymbol{u}_{h},\boldsymbol{u}_{h},\tilde{\boldsymbol{e}}_{i}\right)-b\left(\tilde{\boldsymbol{e}}_{i}, p_{h}\right) =\left(\boldsymbol{f}, \tilde{\boldsymbol{e}}_{i}\right)=\left(\boldsymbol{f}\times\boldsymbol{x},\boldsymbol{e}_{i}\right),
\end{displaymath}
where in the last inequality we also apply \cref{equality1}.
Since $\Delta\boldsymbol{e}_{i}=\Delta\tilde{\boldsymbol{e}}_{i}=\boldsymbol{0}$ and $\nabla\cdot\boldsymbol{e}_{i}=\nabla\cdot\tilde{\boldsymbol{e}}_{i}=0$, it suffices to prove that
\begin{displaymath}
c_{h}\left(\boldsymbol{u}_{h},\boldsymbol{u}_{h},\boldsymbol{e}_{i}\right)=0, \quad c_{h}\left(\boldsymbol{u}_{h},\boldsymbol{u}_{h},\tilde{\boldsymbol{e}}_{i}\right)=0.
\end{displaymath}
By $\nabla\boldsymbol{e}_{i}=\boldsymbol{0}$, \cref{secondessentiallemma} and \cref{assumption2formu} imply that
\begin{displaymath}
c_{h}\left(\boldsymbol{u}_{h},\boldsymbol{u}_{h},\boldsymbol{e}_{i}\right)=-c_{h}\left(\boldsymbol{u}_{h},\boldsymbol{e}_{i},\boldsymbol{u}_{h}\right)
=-c\left(\Pi_{h}\boldsymbol{u}_{h},\boldsymbol{e}_{i},\Pi_{h}\boldsymbol{u}_{h}\right)=0,
\end{displaymath}
and
\begin{displaymath}
c_{h}\left(\boldsymbol{u}_{h},\boldsymbol{u}_{h},\tilde{\boldsymbol{e}}_{i}\right)=-c\left(\Pi_{h}\boldsymbol{u}_{h},\tilde{\boldsymbol{e}}_{i},\Pi_{h}\boldsymbol{u}_{h}\right)
=\left(\Pi_{h}\boldsymbol{u}_{h}\times\Pi_{h}\boldsymbol{u}_{h},\boldsymbol{e}_{i}\right)=0,
\end{displaymath}
together with \cref{equality2}. Thus we complete the proof.
\end{proof}

For the case $\nu=0$ (the Euler equations), if we apply the no-penetration boundary condition ($\boldsymbol{u}\cdot\boldsymbol{n}=0$ on $J\times\Gamma$), the skew-symmetry of $c_{h}$ still holds by \cref{essentiallemma}. Thus \cref{energytheorem} implies that the method \cref{EMAPRsemidiscreteform} conserves a discrete energy for $\nu=0$ and $\boldsymbol{f}=\boldsymbol{0}$, with only no-penetration boundary condition strongly imposed. \cref{MAconservingthm} implies that \cref{EMAPRsemidiscreteform} conserves linear momentum and angular momentum (of $\Pi_{h}\boldsymbol{u}_{h}$) for $\boldsymbol{f}$ with zero momentum and zero angular momentum, respectively (under \cref{assumption2}).
\begin{remark}[Momentum analysis for the rotational form \cref{rotationalform}]\label{remark:momentumforRot}
We use linear momentum as an example. By \cite[Eq. (8)]{Rebholz2017}, we have
\begin{displaymath}
c_{\operatorname{rot}}\left(\boldsymbol{u}_{h},\boldsymbol{u}_{h},\boldsymbol{v}_{h}\right)=
c\left(\Pi_{h}\boldsymbol{u}_{h},\boldsymbol{u}_{h},\Pi_{h}\boldsymbol{v}_{h}\right)-
c\left(\Pi_{h}\boldsymbol{v}_{h},\boldsymbol{u}_{h},\Pi_{h}\boldsymbol{u}_{h}\right).
\end{displaymath}
Then taking $\boldsymbol{v}_{h}=\chi\left(\boldsymbol{e}_{i}\right)$ gives
\begin{displaymath}
\begin{aligned}
c_{\operatorname{rot}}\left(\boldsymbol{u}_{h},\boldsymbol{u}_{h},\boldsymbol{e}_{i}\right)&=
c\left(\Pi_{h}\boldsymbol{u}_{h},\boldsymbol{u}_{h},\boldsymbol{e}_{i}\right)-
c\left(\boldsymbol{e}_{i},\boldsymbol{u}_{h},\Pi_{h}\boldsymbol{u}_{h}\right)\\
&=-c\left(\Pi_{h}\boldsymbol{u}_{h},\boldsymbol{e}_{i},\boldsymbol{u}_{h}\right)-
c\left(\boldsymbol{e}_{i},\boldsymbol{u}_{h},\Pi_{h}\boldsymbol{u}_{h}\right) \quad \text {(by \cref{essentiallemma})}\\
&=-c\left(\boldsymbol{e}_{i},\boldsymbol{u}_{h},\Pi_{h}\boldsymbol{u}_{h}\right)\neq 0.
\end{aligned}
\end{displaymath}
Thus linear momentum is not preserved by $c_{\operatorname{rot}}$. Angular momentum can be similarly checked.
\end{remark}
\section{A pressure-robust and $Re$-semi-robust error estimate}
\label{sec:3}
Let $\boldsymbol{u}$ solve \cref{weakcontinuousform}. We assume $\Delta\boldsymbol{u}\in L^{2}\left(J; \boldsymbol{L}^{2}\left(\Omega\right)\right)$. Multiplying $\Pi_{h}\boldsymbol{v}_{h}\in \Pi_{h}V_{h}^{0}$ on the two sides of \cref{NSE} and integrating over $\Omega$, one arrives at
\begin{equation}\label{continuousform}
\left(\frac{\partial\boldsymbol{u}}{\partial t},\Pi_{h}\boldsymbol{v}_{h}\right)-\nu \left(\Delta\boldsymbol{u}, \Pi_{h}\boldsymbol{v}_{h}\right)+c\left(\boldsymbol{u},\boldsymbol{u},\Pi_{h}\boldsymbol{v}_{h}\right)=\left(\boldsymbol{f}, \Pi_{h}\boldsymbol{v}_{h}\right)
\quad\forall~\boldsymbol{v}_{h}\in V_{h}^{0},
\end{equation}
where the term $\left(\nabla p,\Pi_{h}\boldsymbol{v}_{h}\right)$ has been removed since $\nabla\cdot\Pi_{h}\boldsymbol{v}_{h}\equiv 0$. According to the definition of $\Pi_{h}$, $\Pi_{h}^{1}$ and $\Pi_{h}^{R}$ on the exact solution $\boldsymbol{u}$ (see \Cref{pihontruesolution}) and \cref{dtbilinearform,EMAPRtrilinearform}, \Cref{continuousform} could be rewritten as
\begin{equation}\label{secondcontinuousform}
d_{h}\left(\frac{\partial\boldsymbol{u}}{\partial t},\boldsymbol{v}_{h}\right)-\nu \left(\Delta\boldsymbol{u}, \Pi_{h}\boldsymbol{v}_{h}\right)+c_{h}\left(\boldsymbol{u},\boldsymbol{u},\boldsymbol{v}_{h}\right)=\left(\boldsymbol{f}, \Pi_{h}\boldsymbol{v}_{h}\right)
\quad\forall~\boldsymbol{v}_{h}\in V_{h}^{0}.
\end{equation}

Subtracting \cref{EMAPRsemidiscreteform} from \cref{continuousform} we get the error equation
\begin{equation}\label{errorequation}
\begin{aligned}
d_{h}\left(\frac{\partial\left(\boldsymbol{u}-\boldsymbol{u}_{h}\right)}{\partial t}, \boldsymbol{v}_{h}\right)&+\nu \left(\nabla\left(\boldsymbol{u}-\boldsymbol{u}_{h}\right),\nabla\boldsymbol{v}_{h}\right)+
c_{h}\left(\boldsymbol{u},\boldsymbol{u},\boldsymbol{v}_{h}\right)\\&-c_{h}\left(\boldsymbol{u}_{h},\boldsymbol{u}_{h},\boldsymbol{v}_{h}\right)=\delta_{h}\left(\boldsymbol{u},\boldsymbol{v}_{h}\right)\quad \forall~\boldsymbol{v}_{h}\in V_{h}^{0},
\end{aligned}
\end{equation}
where
\begin{displaymath}
\delta_{h}\left(\boldsymbol{u},\boldsymbol{v}_{h}\right):=\nu\left(\Delta\boldsymbol{u},\Pi_{h}\boldsymbol{v}_{h}\right)+\nu\left(\nabla\boldsymbol{u},\nabla\boldsymbol{v}_{h}\right)=
-\nu\left(\Delta\boldsymbol{u},(1-\Pi_{h})\boldsymbol{v}_{h}\right)
\end{displaymath}
is the consistency error from the diffusion term \cite{linke_robust_2016,linke2016}.

Denote by $\Pi_{h}^{S}: V^{0}\rightarrow V_{h}^{0}$ the Stokes projection which satisfies
\begin{equation}\label{Stokesprojection}
\left(\nabla\left(\boldsymbol{v}-\Pi_{h}^{S}\boldsymbol{v}\right),\nabla\boldsymbol{w}\right)=0 \ \ \ \forall~\boldsymbol{v}\in V^{0}, \boldsymbol{w}\in V_{h}^{0}.
\end{equation}
\begin{assumption}\label{thirdassumption}
For any $\boldsymbol{v}\in V^{0}\cap \boldsymbol{W}^{1,\infty}\left(\Omega\right)$, it holds that
\begin{equation}\label{regularity}
\left\|\nabla\Pi_{h}^{S}\boldsymbol{v}\right\|_{\infty}\leq C \left\|\nabla\boldsymbol{v}\right\|_{\infty}.
\end{equation}
\end{assumption}
We refer the readers to \cite{girault_max-norm_2015} for some analysis of \cref{regularity}.

Split the error $\boldsymbol{u}-\boldsymbol{u}_{h}$ as
\begin{equation}\label{splitting}
\boldsymbol{e}_{h}:=\boldsymbol{u}-\boldsymbol{u}_{h}=\boldsymbol{u}-\Pi_{h}^{S}\boldsymbol{u}+\Pi_{h}^{S}\boldsymbol{u}-\boldsymbol{u}_{h}=\boldsymbol{\eta}+\boldsymbol{\phi}_{h}.
\end{equation}
Next, we introduce the dual norm $\left\|\cdot\right\|_{(V_{h}^{0})^{'}}$ for any linear functional $L$ on $V_{h}^{0}$:
\begin{displaymath}
\left\|L\right\|_{\left(V_{h}^{0}\right)^{'}}:=\sup _{\boldsymbol{v}_{h} \in V_{h}^{0}\setminus\left\{\boldsymbol{0}\right\}} \frac{L\left(\boldsymbol{v}_{h}\right)}{\left\|\nabla \boldsymbol{v}_{h}\right\|},
\end{displaymath}
or for any $\boldsymbol{g}\in \boldsymbol{L}^{2}\left(\Omega\right)$:
\begin{displaymath}
\left\|\boldsymbol{g}\right\|_{\left(V_{h}^{0}\right)^{'}}:=\sup _{\boldsymbol{v}_{h} \in V_{h}^{0}\setminus\left\{\boldsymbol{0}\right\}} \frac{\left(\boldsymbol{g},\boldsymbol{v}_{h}\right)}{\left\|\nabla \boldsymbol{v}_{h}\right\|}.
\end{displaymath}
We also define a mesh-dependent seminorm $|||\cdot|||_{*}$ on $V_{h}$ by
\begin{equation}\label{seminorm}
|||\boldsymbol{w}|||_{*}^{2}:=\sum_{K\in\mathcal{T}_{h}}h_{K}^{-2}\left\|\Pi_{h}^{R}\boldsymbol{w}\right\|_{K}^{2}.
\end{equation}
The following two inequalities will be used to estimate the nonlinear terms.
\begin{lemma}\label{nonlinearestimate}
There exists a positive constant $C$, independent of $h$, such that
\begin{equation}\label{nonlinearestimate1}
c\left(\boldsymbol{z},\boldsymbol{v},\Pi_{h}^{R}\boldsymbol{w}\right)\leq C \left\|\boldsymbol{z}\right\|_{\infty}\left\|\boldsymbol{v}\right\||||\boldsymbol{w}|||_{*},
\end{equation}
and
\begin{equation}\label{nonlinearestimate2}
c\left(\boldsymbol{z},\boldsymbol{v},\Pi_{h}^{R}\boldsymbol{w}\right)\leq C \left\|\boldsymbol{z}\right\|\left\|\boldsymbol{v}\right\|\left|\boldsymbol{w}\right|_{1,\infty},
\end{equation}
for $\left(\boldsymbol{z},\boldsymbol{v},\boldsymbol{w}\right)\in \boldsymbol{L}^{2}\left(\Omega\right)\times V_{h}\times V_{h}^{0}$
with some corresponding regularity conditions.
\end{lemma}
\begin{proof}
It follows from the Schwarz's inequality and the inverse inequality that
\begin{displaymath}
\begin{aligned}
c\left(\boldsymbol{z},\boldsymbol{v},\Pi_{h}^{R}\boldsymbol{w}\right)&=\sum_{K\in\mathcal{T}_{h}}\left(\left(\boldsymbol{z}\cdot\nabla\right)\boldsymbol{v},\Pi_{h}^{R}\boldsymbol{w}\right)_{K}\leq \left\|\boldsymbol{z}\right\|_{\infty}
\sum_{K\in\mathcal{T}_{h}}\left\|\nabla\boldsymbol{v}\right\|_{K}\left\|\Pi_{h}^{R}\boldsymbol{w}\right\|_{K}
\\& \leq \left\|\boldsymbol{z}\right\|_{\infty}\left(\sum_{K\in\mathcal{T}_{h}}h_{K}^{2}\left\|\nabla\boldsymbol{v}\right\|_{K}^{2}\right)^{1/2}\left(\sum_{K\in\mathcal{T}_{h}}h_{K}^{-2}\left\|\Pi_{h}^{R}\boldsymbol{w}\right\|_{K}^{2}\right)^{1/2}\\
&\leq C \left\|\boldsymbol{z}\right\|_{\infty}\left\|\boldsymbol{v}\right\||||\boldsymbol{w}|||_{*}.
\end{aligned}
\end{displaymath}
For the second inequality, from \cref{convergenceassumption} similarly we have
\begin{displaymath}
\begin{aligned}
c\left(\boldsymbol{z},\boldsymbol{v},\Pi_{h}^{R}\boldsymbol{w}\right)&=\sum_{K\in\mathcal{T}_{h}}\left(\left(\boldsymbol{z}\cdot\nabla\right)\boldsymbol{v},\Pi_{h}^{R}\boldsymbol{w}\right)_{K}\leq
\sum_{K\in\mathcal{T}_{h}}\left\|\boldsymbol{z}\right\|_{K}\left\|\nabla\boldsymbol{v}\right\|_{K}\left\|\Pi_{h}^{R}\boldsymbol{w}\right\|_{\infty, K}
\\& \leq \sum_{K\in\mathcal{T}_{h}}\left\|\boldsymbol{z}\right\|_{K}h_{K}\left\|\nabla\boldsymbol{v}\right\|_{K}h_{K}^{-1}\left\|\Pi_{h}^{R}\boldsymbol{w}\right\|_{\infty, K}
\\& \leq C \sum_{K\in\mathcal{T}_{h}}\left\|\boldsymbol{z}\right\|_{K}\left\|\boldsymbol{v}\right\|_{K}\left|\boldsymbol{w}\right|_{1,\infty,K}
\leq C \left\|\boldsymbol{z}\right\|\left\|\boldsymbol{v}\right\|\left|\boldsymbol{w}\right|_{1,\infty}.
\end{aligned}
\end{displaymath}
This completes the proof.
\end{proof}
\begin{theorem}\label{mainresulttheorem}
Let $\boldsymbol{u}$ be the solution of \cref{weakcontinuousform} and $\boldsymbol{u}_{h}$ be the solution of \cref{EMAPRsemidiscreteform}. Under \cref{assumption1}, \cref{thirdassumption} and the assumptions that $\boldsymbol{u}\in L^{2}\left(J; \boldsymbol{W}^{1,\infty}\left(\Omega\right)\right)\cap L^{4}\left(J; \boldsymbol{H}^{1}\left(\Omega\right)\right)$, $\boldsymbol{u}_{t}\in L^{2}\left(J; \boldsymbol{L}^{2}\left(\Omega\right)\right)$,
 $\Delta\boldsymbol{u}\in L^{2}\left(J; \boldsymbol{L}^{2}\left(\Omega\right)\right)$ and $\boldsymbol{u}_{h}^{0}=\Pi_{h}^{S}\boldsymbol{u}^{0}$, with $C$ independent of $h$ and $\nu$ the following estimate holds:
\begin{equation}\label{mainresult}
\begin{aligned}
E_{d}(\boldsymbol{e}_{h}&\left(T\right))+\frac{\nu}{2}\int_{0}^{T}\left\|\nabla\boldsymbol{e}_{h}\right\|^{2} ~dt\leq E_{d}\left(\boldsymbol{\eta}\left(T\right)\right)+\frac{\nu}{2}\int_{0}^{T}\left\|\nabla\boldsymbol{\eta}\right\|^{2} ~dt\\&+e^{{G\left(\boldsymbol{u},T\right)}}\int_{0}^{T}\left.\bigg\{ E_{d}\left(\boldsymbol{\eta}_{t}\right)+\frac{\nu}{2}\left\|\Delta \boldsymbol{u} \circ(1-\Pi_{h})\right\|_{\left(V_{h}^{0}\right)^{'}}^{2}\right.
\\&\left.+C\left\|\boldsymbol{u}\right\|_{1,\infty}\left(\left\|\Pi_{h}\boldsymbol{\eta}\right\|^{2}+\left\|\nabla\Pi_{h}^{1}\boldsymbol{\eta}\right\|^{2}+|||\boldsymbol{\eta}|||_{*}^{2}\right)\right.\bigg\}~ dt,
\end{aligned}
\end{equation}
where $G(\boldsymbol{u},T)=T+C_{\alpha}\left\|\boldsymbol{u}\right\|_{L^{1}\left(J;\boldsymbol{W}^{1,\infty}\left(\Omega\right)\right)}$ with $C_{\alpha}$ dependent on $\alpha$ and the shape regularity of mesh but independent of $h$ and $\nu$.
\end{theorem}
\begin{proof}
Substituting \cref{splitting} into \cref{errorequation} and taking $\boldsymbol{v}_{h}=\boldsymbol{\phi}_{h}$ give that
\begin{equation}\label{realerroreq}
\begin{aligned}
\frac{d}{dt}E_{d}\left(\boldsymbol{\phi}_{h}\right)&+\nu\left\|\nabla\boldsymbol{\phi}_{h}\right\|^{2}=d_{h}\left(\boldsymbol{\eta}_{t},\boldsymbol{\phi}_{h}\right)-\underbrace{\nu\left(\nabla\boldsymbol{\eta},\nabla\boldsymbol{\phi}_{h}\right)}_{=0 \text{ by } \cref{Stokesprojection}}\\
&\underbrace{-c_{h}\left(\boldsymbol{u},\boldsymbol{u},\boldsymbol{\phi}_{h}\right)+c_{h}\left(\boldsymbol{u}_{h},\boldsymbol{u}_{h},\boldsymbol{\phi}_{h}\right)}_{\mathcal{NL}}+\delta_{h}\left(\boldsymbol{u},\boldsymbol{\phi}_{h}\right).
\end{aligned}
\end{equation}
Let us estimate each term in \cref{realerroreq}. For the evolutionary term we have
\begin{equation}\label{estimate1}
\left|d_{h}\left(\boldsymbol{\eta}_{t},\boldsymbol{\phi}_{h}\right)\right|\leq E_{d}\left(\boldsymbol{\eta}_{t}\right)+E_{d}\left(\boldsymbol{\phi}_{h}\right).
\end{equation}
For $\delta_{h}$ the estimate could be found in \cite{linke2016}:
\begin{equation}\label{estimate2}
\left|\delta_{h}\left(\boldsymbol{u},\boldsymbol{\phi}_{h}\right)\right|
\leq \frac{1}{2}\nu\left\|\Delta \boldsymbol{u} \circ\left(1-\Pi_{h}\right)\right\|_{\left(V_{h}^{0}\right)^{'}}^{2}+\frac{1}{2}\nu\left\|\nabla \boldsymbol{\phi}_{h}\right\|^{2}.
\end{equation}
Now, let us estimate the convective terms. We use a similar decomposition with \cite{schroeder_towards_2018,Rebholz2020}:
\begin{equation}\label{estimate3}
-\mathcal{NL}=c_{h}\left(\boldsymbol{u},\boldsymbol{\eta},\boldsymbol{\phi}_{h}\right)
+c_{h}\left(\boldsymbol{u},\Pi_{h}^{S}\boldsymbol{u},\boldsymbol{\phi}_{h}\right)
-c_{h}\left(\boldsymbol{u}_{h},\boldsymbol{u}_{h},\boldsymbol{\phi}_{h}\right).
\end{equation}
Further,
\begin{equation}\label{estimate4}
\begin{aligned}
c_{h}\left(\boldsymbol{u},\Pi_{h}^{S}\boldsymbol{u},\boldsymbol{\phi}_{h}\right)
-c_{h}\left(\boldsymbol{u}_{h},\boldsymbol{u}_{h},\boldsymbol{\phi}_{h}\right)&=
c_{h}\left(\boldsymbol{e}_{h},\Pi_{h}^{S}\boldsymbol{u},\boldsymbol{\phi}_{h}\right)
+c_{h}\left(\boldsymbol{u}_{h},\Pi_{h}^{S}\boldsymbol{u},\boldsymbol{\phi}_{h}\right)\\
&-c_{h}\left(\boldsymbol{u}_{h},\boldsymbol{u}_{h},\boldsymbol{\phi}_{h}\right)\\&=
c_{h}\left(\boldsymbol{e}_{h},\Pi_{h}^{S}\boldsymbol{u},\boldsymbol{\phi}_{h}\right)
+c_{h}\left(\boldsymbol{u}_{h},\boldsymbol{\phi}_{h},\boldsymbol{\phi}_{h}\right)\\&=
c_{h}\left(\boldsymbol{e}_{h},\Pi_{h}^{S}\boldsymbol{u},\boldsymbol{\phi}_{h}\right)\\
& = c_{h}\left(\boldsymbol{\eta},\Pi_{h}^{S}\boldsymbol{u},\boldsymbol{\phi}_{h}\right)+
c_{h}\left(\boldsymbol{\phi}_{h},\Pi_{h}^{S}\boldsymbol{u},\boldsymbol{\phi}_{h}\right).
\end{aligned}
\end{equation}
Then it follows from the Schwarz's inequality, Young's inequality, \cref{nonlinearestimate}, \cref{assumption1} and \cref{thirdassumption} that
\begin{equation}\label{estimate5}
\begin{aligned}
|c_{h}(\boldsymbol{u},&\boldsymbol{\eta},\boldsymbol{\phi}_{h})|=\left|c\left(\boldsymbol{u},\Pi_{h}^{1}\boldsymbol{\eta},\Pi_{h}\boldsymbol{\phi}_{h}\right)
-c\left(\boldsymbol{u},\Pi_{h}^{1}\boldsymbol{\phi}_{h},{\Pi_{h}^{R}\boldsymbol{\eta}}\right)\right|_{{\left(\Pi_{h}^{R}\boldsymbol{\eta}=-\Pi_{h}^{R}\Pi_{h}^{S}\boldsymbol{u}\right)}}\\
&\leq \left\|\boldsymbol{u}\right\|_{\infty}\left(\left\|\nabla\Pi_{h}^{1}\boldsymbol{\eta}\right\|\left\|\Pi_{h}\boldsymbol{\phi}_{h}\right\|+
C \left\|\Pi_{h}^{1}\boldsymbol{\phi}_{h}\right\||||\Pi_{h}^{S}\boldsymbol{u}|||_{*}\right)\\
&\leq \left\|\boldsymbol{u}\right\|_{\infty}\left(\frac{1}{2}\left\|\nabla\Pi_{h}^{1}\boldsymbol{\eta}\right\|^{2}+\frac{1}{2}\left\|\Pi_{h}\boldsymbol{\phi}_{h}\right\|^{2}+
C \left(\left\|\Pi_{h}^{1}\boldsymbol{\phi}_{h}\right\|^{2}+|||\boldsymbol{\eta}|||_{*}^{2}\right)\right),
\end{aligned}
\end{equation}
\begin{equation}\label{estimate6}
\begin{aligned}
\left|c_{h}\left(\boldsymbol{\eta},\Pi_{h}^{S}\boldsymbol{u},\boldsymbol{\phi}_{h}\right)\right|&=\left|c\left(\Pi_{h}\boldsymbol{\eta},\Pi_{h}^{1}\Pi_{h}^{S}\boldsymbol{u},\Pi_{h}\boldsymbol{\phi}_{h}\right)-
c\left(\Pi_{h}\boldsymbol{\eta},\Pi_{h}^{1}\boldsymbol{\phi}_{h},\Pi_{h}^{R}\Pi_{h}^{S}\boldsymbol{u}\right)\right|\\
& \leq C \left|\boldsymbol{u}\right|_{1,\infty}\left\|\Pi_{h}\boldsymbol{\eta}\right\|\left(\left\|\Pi_{h}\boldsymbol{\phi}_{h}\right\|+\left\|\Pi_{h}^{1}\boldsymbol{\phi}_{h}\right\|\right)\\
& \leq C \left|\boldsymbol{u}\right|_{1,\infty}\left(\left\|\Pi_{h}\boldsymbol{\eta}\right\|^{2}+
\left\|\Pi_{h}\boldsymbol{\phi}_{h}\right\|^{2}+\left\|\Pi_{h}^{1}\boldsymbol{\phi}_{h}\right\|^{2}\right),
\end{aligned}
\end{equation}
and
\begin{equation}\label{estimate7}
\begin{aligned}
&\left|c_{h}\left(\boldsymbol{\phi}_{h},\Pi_{h}^{S}\boldsymbol{u},\boldsymbol{\phi}_{h}\right)\right|=
\big|c\left(\Pi_{h}\boldsymbol{\phi}_{h},\Pi_{h}^{1}\Pi_{h}^{S}\boldsymbol{u},\Pi_{h}\boldsymbol{\phi}_{h}\right)\\
&-c\left(\Pi_{h}\boldsymbol{\phi}_{h},\Pi_{h}^{1}\boldsymbol{\phi}_{h},\Pi_{h}^{R}\Pi_{h}^{S}\boldsymbol{u}\right)\big|
 \leq C \left|\boldsymbol{u}\right|_{1,\infty}\left(\left\|\Pi_{h}\boldsymbol{\phi}_{h}\right\|^{2}+\left\|\Pi_{h}^{1}\boldsymbol{\phi}_{h}\right\|^{2}\right).
\end{aligned}
\end{equation}
Substituting \crefrange{estimate4}{estimate7} into \cref{estimate3} one could obtain that
\begin{equation}\label{estimate8}
\left|\mathcal{NL}\right|\leq C \left\|\boldsymbol{u}\right\|_{1,\infty}\left(\left\|\Pi_{h}\boldsymbol{\eta}\right\|^{2}+\left\|\nabla\Pi_{h}^{1}\boldsymbol{\eta}\right\|^{2}+|||\boldsymbol{\eta}|||_{*}^{2}+
\left\|\Pi_{h}\boldsymbol{\phi}_{h}\right\|^{2}+\left\|\Pi_{h}^{1}\boldsymbol{\phi}_{h}\right\|^{2}\right).
\end{equation}
Then substituting \cref{estimate1,estimate2,estimate8} into \cref{realerroreq} provides
\begin{equation}\label{mainresults0}
\begin{aligned}
\frac{d}{dt}&E_{d}\left(\boldsymbol{\phi}_{h}\right)+\frac{\nu}{2}\left\|\nabla\boldsymbol{\phi}_{h}\right\|^{2}\leq E_{d}\left(\boldsymbol{\eta}_{t}\right)+E_{d}\left(\boldsymbol{\phi}_{h}\right)+\frac{1}{2}\nu\left\|\Delta \boldsymbol{u} \circ\left(1-\Pi_{h}\right)\right\|_{\left(V_{h}^{0}\right)^{'}}^{2}\\
&+C\left\|\boldsymbol{u}\right\|_{1,\infty}\left(\left\|\Pi_{h}\boldsymbol{\eta}\right\|^{2}+\left\|\nabla\Pi_{h}^{1}\boldsymbol{\eta}\right\|^{2}+|||\boldsymbol{\eta}|||_{*}^{2}+
\left\|\Pi_{h}\boldsymbol{\phi}_{h}\right\|^{2}+\left\|\Pi_{h}^{1}\boldsymbol{\phi}_{h}\right\|^{2}\right).
\end{aligned}
\end{equation}
Note that $\left\|\Pi_{h}\boldsymbol{\phi}_{h}\right\|^{2}\leq 2 E_{d}\left(\boldsymbol{\phi}_{h}\right)$ and $\left\|\Pi_{h}^{1}\boldsymbol{\phi}_{h}\right\|^{2}\leq \left(4+4/\alpha\right) E_{d}\left(\boldsymbol{\phi}_{h}\right)$ by \cite[Lemma 3]{brezzi_mixed_2005}. Finally, integrating over $J$, and applying the fact $\boldsymbol{u}_{h}\left(0\right)=\Pi_{h}^{S}\boldsymbol{u}\left(0\right)$ and the Gronwall inequality, we can get
\begin{equation}\nonumber
\begin{aligned}
E_{d}\left(\boldsymbol{\phi}_{h}\left(T\right)\right)&+\frac{\nu}{2}\int_{0}^{T}\left\|\nabla\boldsymbol{\phi}_{h}\right\|^{2} ~dt \leq e^{{G\left(\boldsymbol{u},T\right)}}\int_{0}^{T}\left.\bigg\{ E_{d}\left(\boldsymbol{\eta}_{t}\right)+\frac{\nu}{2}\left\|\Delta \boldsymbol{u} \circ\left(1-\Pi_{h}\right)\right\|_{\left(V_{h}^{0}\right)^{'}}^{2}\right.
\\&\left.+C\left\|\boldsymbol{u}\right\|_{1,\infty}\left(\left\|\Pi_{h}\boldsymbol{\eta}\right\|^{2}+\left\|\nabla\Pi_{h}^{1}\boldsymbol{\eta}\right\|^{2}+|||\boldsymbol{\eta}|||_{*}^{2}\right)\right.\bigg\} ~dt.
\end{aligned}
\end{equation}
Then \cref{mainresult} follows immediately from a combination of the above inequality and the triangle inequality.
\end{proof}
\section{The reconstruction on simiplicial locally mass-conserving elements}
\label{sec:4}
In this section, we focus on a class of simiplicial locally divergence-free elements which satisfy the inf-sup condition \cref{infsupcondition}, and give the corresponding divergence-free reconstruction operators. First, let us recall the construction of the locally divergence-free elements in \cite[pp. 132-144]{girault_finite_1986}, where the lowest order case is the well-known Bernardi-Raugel element \cite{bernardi_analysis_1985}. Consider an arbitrary element $K$ with vertices $a_{i},1\leq i\leq d+1$. Denote by $e_{i}$ the edge/face opposite to $a_{i}$ and $\boldsymbol{n}_{i}$ the unit outward normal vector corresponding to $e_{i}, 1\leq i\leq d+1$. Further, $\lambda_{i}, 1\leq i\leq d+1$ denote the corresponding barycentric coordinates. Then the face bubbles are defined by
\begin{displaymath}
\boldsymbol{b}_{i}:=\left(\prod_{1\leq j\leq d+1; j\neq i}\lambda_{j}\right)\boldsymbol{n}_{i}, 1\leq i \leq d+1.
\end{displaymath}
We also define
\begin{displaymath}
{b}_{K}:=\prod_{1\leq j\leq d+1}\lambda_{j},\quad \tilde{P}^{k}\left(K\right):=\operatorname{span}\left\{\prod_{i=1}^{d}x_{i}^{k_{i}}, k_{i}\geq 0, 1\leq i\leq d, \sum_{i=1}^{d}k_{i}=k\right\}.
\end{displaymath}
Then the local finite element spaces for velocity on an element $K$ are defined as \cref{tab:elements}.
\begin{table}
\centering
\caption{Local velocity spaces on an element $K$.}
\begin{tabular}{|c|c|l|}
  \hline
  Order    & Dimension     & Local space \\
  \hline
  $k=1$    &  2D/3D        & $\boldsymbol{{\mathcal{P}}}_{k}\left(K\right)=\left[P^{1}\left(K\right)\right]^{d}\oplus\operatorname{span}\left\{\boldsymbol{b}_{i},1\leq i\leq d+1\right\}$\\
  \hline
  $k\geq2$ &  2D           & $\boldsymbol{{\mathcal{P}}}_{k}\left(K\right)=\left[P^{k}\left(K\right)\oplus{b}_{K}\tilde{P}^{k-2}\left(K\right)\right]^{2}$\\
  \hline
  $k=2$    &  3D           & $\boldsymbol{{\mathcal{P}}}_{k}\left(K\right)=\left[P^{2}\left(K\right)\oplus{b}_{K}\tilde{P}^{0}\left(K\right)\right]^{3}\oplus\operatorname{span}\left\{\boldsymbol{b}_{i},1\leq i\leq 4\right\}$\\
  \hline
  $k\geq3$ &  3D           & $\boldsymbol{{\mathcal{P}}}_{k}\left(K\right)=\left[P^{k}\left(K\right)\oplus{b}_{K}\left(\tilde{P}^{k-2}\left(K\right)\oplus \tilde{P}^{k-3}\left(K\right)\right)\right]^{3}$\\
  \hline
\end{tabular}
\label{tab:elements}
\end{table}
For $k$-th order velocity spaces, the matching pressure space is the space of discontinuous piecewise polynomials of degree no more than $k-1$, whatever the dimension is. In what follows, $V_{h}$ will denote a velocity space of order $k$ mentioned above and $W_{h}$ is the corresponding pressure space. From \cref{tab:elements} one can see there exists a space of bubble functions $V_{h}^{b}\cong V_{h}/\boldsymbol{P}^{k}$ such that $V_{h}=\boldsymbol{P}^{k}\oplus V_{h}^{b}$. For any $\boldsymbol{v}_{h}\in V_{h}$, it is natural to split it into two parts:
\begin{displaymath}
\boldsymbol{v}_{h}=\boldsymbol{v}_{h}^{1}+\boldsymbol{v}_{h}^{b} \text{ with } \boldsymbol{v}_{h}^{1}\in \boldsymbol{P}^{k}, \boldsymbol{v}_{h}^{b} \in  V_{h}^{b}.
\end{displaymath}
We consider a class of divergence-free reconstruction operators which were also discussed in \cite[Remark 4.2]{linke2016}. These operators are defined as follows. Let $I_{h}: \boldsymbol{C}^{0}\left(\bar{\Omega}\right)\rightarrow \boldsymbol{P}^{k}$ denote the usual nodal interpolation operator. Then $\Pi_{h}^{1}$ is defined by
\begin{displaymath}
\Pi_{h}^{1}|_{V_{h}}:=I_{h}.
\end{displaymath}
 Let $\Pi_{h}^{RT}$ be the common Raviart-Thomas interpolation of order $k-1$ \cite{boffi_mixed_2013}. The operator $\Pi_{h}^{R}$ is defined by
\begin{equation}\label{realpihr}
\Pi_{h}^{R}|_{V_{h}}:=\Pi_{h}^{RT}\circ\left(1-I_{h}\right).
\end{equation}
Hence, for any $\boldsymbol{v}_{h}\in V_{h}$ one has
\begin{equation}\label{realpih}
\Pi_{h}\boldsymbol{v}_{h}=\Pi_{h}^{1}\boldsymbol{v}_{h}+\Pi_{h}^{R}\boldsymbol{v}_{h}=I_{h}\boldsymbol{v}_{h}+\Pi_{h}^{RT}\left(\boldsymbol{v}_{h}-I_{h}\boldsymbol{v}_{h}\right).
\end{equation}
At this time the space $X_{h}$ could be chosen as
\begin{displaymath}
X_{h}:=\left\{\boldsymbol{v}_{h}\in \boldsymbol{H}\left(\operatorname{div};\Omega\right): \boldsymbol{v}_{h}|_{K}\in \left[P^{k}\left(K\right)\right]^{d}\text{ for all } K\in\mathcal{T}_{h}\right\}.
\end{displaymath}
Clearly $X_{h}$ and $W_{h}$ satisfy the relationship \cref{divfreeinclude}.
\begin{remark}
Under the setting above, for any $\boldsymbol{v}_{h}\in V_{h}$ it holds that
$\Pi_{h}\boldsymbol{v}_{h}=\boldsymbol{v}_{h}^{1}+\Pi_{h}\boldsymbol{v}_{h}^{b}.$ In other words, these reconstruction operators only change the bubble part of the elements. Thus the reconstruction is low-cost and has not changed much compared to the previous classical formulation. This is especially the case for the first order element and the second order element in two dimensions, and the first order element in three dimensions, since it is not hard to find that, for any $\boldsymbol{v}_{h}$ belonging to these spaces,
\begin{displaymath}
\Pi_{h}^{1}\boldsymbol{v}_{h}^{b}=0\Rightarrow \Pi_{h}\boldsymbol{v}_{h}^{b}=\Pi_{h}^{R}\boldsymbol{v}_{h}^{b}=\Pi_{h}^{RT}\boldsymbol{v}_{h}^{b}\Rightarrow \Pi_{h}\boldsymbol{v}_{h}=\boldsymbol{v}_{h}^{1}+\Pi_{h}^{RT}\boldsymbol{v}_{h}^{b}.
\end{displaymath}
\end{remark}
Next, let us analyze the properties of the reconstruction operators defined above. To analyze the convergence rates of EMAPR for the elements mentioned above, we shall assume that the true solution $\boldsymbol{u}\left(t\right)$ ($t\in J$) is in $\boldsymbol{H}^{\frac{3}{2}+\epsilon}\left(\Omega\right)$ with some $\epsilon>0$. This assumption guarantees $\boldsymbol{u}\left(t\right)\in \boldsymbol{C}^{0}\left(\bar{\Omega}\right)$ and thus $I_{h}\boldsymbol{u}$ is well-defined. We define $\boldsymbol{V}\left(\mathcal{T}_{h}\right):=\boldsymbol{C}^{0}\left(\bar{\Omega}\right)\cap \boldsymbol{H}^{2}\left(\mathcal{T}_{h}\right)$.
\begin{lemma}\label{essentiallemma2}
The operators $I_{h}$ and $\Pi_{h}^{RT}$ satisfy the following properties:
\begin{equation}\label{Ihestimate1}
\left\|\boldsymbol{v}-I_{h}\boldsymbol{v}\right\|_{K}+h_{K}\left\|\nabla(\boldsymbol{v}-I_{h}\boldsymbol{v})\right\|_{K}\leq C h_{K}^{2} \left|\boldsymbol{v}\right|_{2,K} \quad\forall \ \boldsymbol{v}\in \boldsymbol{V}\left(\mathcal{T}_{h}\right) ;
\end{equation}
\begin{equation}\label{Ihestimate2}
\left\|\boldsymbol{v}-I_{h}\boldsymbol{v}\right\|_{K}+h_{K}\left\|\nabla(\boldsymbol{v}-I_{h}\boldsymbol{v})\right\|_{K}\leq C h_{K}^{2} \left|\boldsymbol{v}-I_{h}\boldsymbol{v}\right|_{2,K}  \quad\forall \ \boldsymbol{v}\in \boldsymbol{V}\left(\mathcal{T}_{h}\right) ;
\end{equation}
\begin{equation}\label{pihrtestimate1}
\left\|\boldsymbol{v}-\Pi_{h}^{RT}\boldsymbol{v}\right\|_{K}\leq C h_{K}\left\|\nabla\boldsymbol{v}\right\|_{K} \quad\forall \ \boldsymbol{v}\in V;
\end{equation}
\begin{equation}\label{pihrtestimate2}
\left\|\Pi_{h}^{RT}\boldsymbol{v}_{h}\right\|_{p,K}\leq \left\|\boldsymbol{v}_{h}\right\|_{p,K},\quad p=2,\infty, \quad \forall\  \boldsymbol{v}_{h}\in V_{h};
\end{equation}
\begin{equation}\label{rtconsistency}
\left(\boldsymbol{v}-\Pi_{h}^{RT} \boldsymbol{v}, \boldsymbol{w}\right)_{K}=0 \quad \forall~ \boldsymbol{v}\in V, \boldsymbol{w} \in \left[P^{k-2}(K)\right]^{d},
\end{equation}
for all $K\in\mathcal{T}_{h}$.
\end{lemma}
\begin{proof}
We refer the readers to \cite[Theorem 4.4.4]{brenner_mathematical_2008} and \cite[Propositions 2.5.1, 2.3.4]{boffi_mixed_2013} for \cref{Ihestimate1}, \cref{pihrtestimate1} and \cref{rtconsistency}, respectively. Note that $I_{h}\left(\boldsymbol{v}-I_{h}\boldsymbol{v}\right)=\boldsymbol{0}$. Then replacing $\boldsymbol{v}$ with $\boldsymbol{v}-I_{h}\boldsymbol{v}$ in \cref{Ihestimate1} gives \cref{Ihestimate2}.

Let us prove \cref{pihrtestimate2}. We have
\begin{equation}\nonumber
\begin{aligned}
\left\|\boldsymbol{v}_{h}-\Pi_{h}^{RT}\boldsymbol{v}_{h}\right\|_{p,K}&\leq C h_{K}^{d/p-d/2}\left\|\boldsymbol{v}_{h}-\Pi_{h}^{RT}\boldsymbol{v}_{h}\right\|_{K}\\
&\leq C h_{K}^{d/p-d/2+1}\left\|\nabla\boldsymbol{v}_{h}\right\|_{K}\leq C \left\|\boldsymbol{v}_{h}\right\|_{p,K},
\end{aligned}
\end{equation}
where we repeatedly use the local estimates in \cite[Lemma 4.5.3]{brenner_mathematical_2008} and the interpolation error of $\Pi_{h}^{RT}$ \cref{pihrtestimate1}. Then \cref{pihrtestimate2} follows immediately from the triangle inequality. Note that \cref{pihrtestimate2} does not hold if $\boldsymbol{v}_{h}$ is an arbitrary function in $V$.
\end{proof}
\begin{lemma}
The reconstruction operators defined by \cref{realpih} satisfy \crefrange{divfreedefinition}{convergenceassumption} and \cref{assumption2formu}.
\end{lemma}
\begin{proof}
\Cref{assumption2formu} is clearly satisfied. Furthermore since \cref{divfreedefinition} is implied in \cref{bconsistency} due to $\nabla\cdot\Pi_{h}\boldsymbol{v}_{h}\in W_{h}$ for any $\boldsymbol{v}_{h}\in V_{h}$ (see \cref{divfreeinclude}), we only prove \crefrange{bconsistency}{convergenceassumption}.

Denote by $P_{h}: L^{2}\left(\Omega\right)\rightarrow W_{h}$ the $L^{2}$ projection to $W_{h}$. Applying the commuting diagram property for $\Pi_{h}^{RT}$ and $P_{h}$ (e.g., see \cite[Remark 2.5.2]{boffi_mixed_2013}) and \cref{divfreeinclude} one can obtain
\begin{displaymath}
\begin{aligned}
P_{h}\nabla\cdot\Pi_{h}\boldsymbol{v}_{h}&=\nabla\cdot\Pi_{h}\boldsymbol{v}_{h}=
\nabla\cdot\Pi_{h}^{1}\boldsymbol{v}_{h}+\nabla\cdot\Pi_{h}^{RT}\left(1-\Pi_{h}^{1}\right)\boldsymbol{v}_{h}\\&
=\nabla\cdot\Pi_{h}^{1}\boldsymbol{v}_{h}+
P_{h}\nabla\cdot\left(1-\Pi_{h}^{1}\right)\boldsymbol{v}_{h}=P_{h}\nabla\cdot\boldsymbol{v}_{h},
\end{aligned}
\end{displaymath}
which is exactly \cref{bconsistency}.

The proof of \cref{consistency} is very similar to the analysis in \cite{linke_robust_2016}. In fact, for any $\boldsymbol{v}_{h}\in V_{h}$, from \cref{realpih} we have
\begin{equation}\label{vminuspihv}
\boldsymbol{v}_{h}-\Pi_{h}\boldsymbol{v}_{h}=\left(\boldsymbol{v}_{h}-I_{h}\boldsymbol{v}_{h}\right)-\Pi_{h}^{RT}\left(\boldsymbol{v}_{h}-I_{h}\boldsymbol{v}_{h}\right),
\end{equation}
which, together with \cref{rtconsistency}, implies that
\begin{equation}\label{consistency2}
\left(\boldsymbol{v}_{h}-\Pi_{h} \boldsymbol{v}_{h}, \boldsymbol{w}\right)_{K}=0 \quad \forall~ \boldsymbol{w} \in \left[P^{k-2}(K)\right]^{d}, K \in \mathcal{T}_{h}.
\end{equation}
On the other hand, from \cref{vminuspihv}, \cref{pihrtestimate2}, \cref{Ihestimate1} and inverse inequalities, it follows that
\begin{equation}\label{consistency3}
\left\|\boldsymbol{v}_{h}-\Pi_{h}\boldsymbol{v}_{h}\right\|_{K}\leq C \left\|(1-I_{h})\boldsymbol{v}_{h}\right\|_{K}\leq C h_{K}^{2}\left|\boldsymbol{v}_{h}\right|_{2,K}\leq C h_{K}\left\|\nabla\boldsymbol{v}_{h}\right\|_{K}.
\end{equation}
Then a combination of \cref{consistency2}, \cref{consistency3} and approximation theory gives
\begin{displaymath}
\left|\left(\boldsymbol{g},\left(1-\Pi_{h}\right)\boldsymbol{v}_{h}\right)\right|=\left|\left(\boldsymbol{g}-P_{h}^{k-2}\boldsymbol{g},\left(1-\Pi_{h}\right)\boldsymbol{v}_{h}\right)\right|\leq C h^{k}\left|\boldsymbol{g}\right|_{k-1}\left\|\nabla\boldsymbol{v}_{h}\right\|,
\end{displaymath}
where $P_{h}^{k-2}$ is the $L^{2}$ projection operator to the space of piecewise polynomials of degree no more than $k-2$. This completes the proof of \cref{consistency}.

By \cite[Theorem 4.4.4]{brenner_mathematical_2008} we have
\begin{equation}\label{inftyapproximation}
\left\|\left(1-I_{h}\right)\boldsymbol{v}_{h}\right\|_{\infty,K}+h_{K}\left|\left(1-I_{h}\right)\boldsymbol{v}_{h}\right|_{1,\infty,K}\leq C h_{K} \left|\boldsymbol{v}_{h}\right|_{1,\infty,K},
\end{equation}
for all $K\in \mathcal{T}_{h}, \boldsymbol{v}_{h}\in V_{h}^{0}$.
The inequality \cref{boundassumption} follows immediately from a combination of the triangle inequality and \cref{inftyapproximation}.

For \cref{convergenceassumption}, similarly to \cref{consistency3}, using \cref{inftyapproximation} we have
\begin{displaymath}
\left\|\Pi_{h}^{R}\boldsymbol{v}_{h}\right\|_{\infty,K}\leq C \left\|\left(1-I_{h}\right)\boldsymbol{v}_{h}\right\|_{\infty,K}\leq C h_{K} \left|\boldsymbol{v}_{h}\right|_{1,\infty,K}.
\end{displaymath}
Thus we complete the proof.
\end{proof}

Up to now, we have proven that a class of divergence-free reconstruction operators satisfy the assumptions in \cref{sec:21}. Thus it admits an a priori error estimate in \cref{mainresulttheorem}. The only question is whether the right-hand side of \cref{mainresult} is an $O\left(h^{k}\right)$ quantity if $\boldsymbol{u}$ is sufficiently smooth. The non-trivial terms are the ones corresponding to $\Pi_{h}\boldsymbol{\eta}$ (including $E_{d}\left(\boldsymbol{\eta}\right)$), $\Pi_{h}^{1}\boldsymbol{\eta}$, $\Pi_{h}^{R}\boldsymbol{\eta}$ (including $E_{d}\left(\boldsymbol{\eta}\right)$ and $|||\boldsymbol{\eta}|||_{*}$) and $\left\|\Delta \boldsymbol{u} \circ\left(1-\Pi_{h}\right)\right\|_{\left(V_{h}^{0}\right)^{'}}$. The following lemmas are to answer this question.
\begin{lemma}\label{consistencylemma}
Let $\boldsymbol{u}$ be the solution of \cref{weakcontinuousform}. Suppose that $\boldsymbol{u}\left(t\right)\in \boldsymbol{H}^{k+1}\left(\Omega\right)$ $(t\in J)$. Then we have
\begin{equation}\label{conssistencyestimate}
\left\|\Delta \boldsymbol{u}\left(t\right) \circ\left(1-\Pi_{h}\right)\right\|_{\left(V_{h}^{0}\right)^{'}}\leq C h^{k} \left|\boldsymbol{u}\left(t\right)\right|_{k+1}.
\end{equation}
\end{lemma}
\begin{proof}
The inequality \cref{conssistencyestimate} follows immediately from \cref{consistency} by taking $\boldsymbol{g}=\Delta\boldsymbol{u}\left(t\right)$.
\end{proof}
Introduce the seminorm $|||\cdot|||_{2}$ on $\boldsymbol{H}^{2}\left(\mathcal{T}_{h}\right)$ by
\begin{displaymath}
|||\boldsymbol{v}|||_{2}^{2}:=\sum_{K\in\mathcal{T}_{h}} \left|\boldsymbol{v}\right|_{2,K}^{2} \quad\forall~ \boldsymbol{v}\in \boldsymbol{H}^{2}\left(\mathcal{T}_{h}\right).
\end{displaymath}
\begin{lemma}\label{piestimatelemma}
Let $\boldsymbol{u}$ be the solution of \cref{weakcontinuousform} and $\boldsymbol{\eta}=\boldsymbol{u}-\Pi_{h}^{S}\boldsymbol{u}$. Suppose that $\boldsymbol{u}\in \boldsymbol{H}^{\frac{3}{2}+\epsilon}\left(\Omega\right)\cap \boldsymbol{H}^{2}\left(\mathcal{T}_{h}\right)$ with $\epsilon>0$. Then we have
\begin{equation}\label{piestimate}
\left\|\Pi_{h}\boldsymbol{\eta}\right\|\leq \left\|\boldsymbol{\eta}\right\|+ C \left(\left\|\boldsymbol{u}-I_{h}\boldsymbol{u}\right\|+\left\|\boldsymbol{\eta}\right\|+ h^{2} |||\boldsymbol{\eta}|||_{2}\right),
\end{equation}
\begin{equation}\label{piestimate1}
\left\|\nabla\Pi_{h}^{1}\boldsymbol{\eta}\right\|\leq 2 \left\|\nabla\boldsymbol{\eta}\right\|+ \left\|\nabla\left(\boldsymbol{u}-I_{h}\boldsymbol{u}\right)\right\|+C h |||\boldsymbol{\eta}|||_{2},
\end{equation}
\begin{equation}\label{piestimater}
\left\|\Pi_{h}^{R}\boldsymbol{\eta}\right\|\leq C \left(\left\|\boldsymbol{u}-I_{h}\boldsymbol{u}\right\|+\left\|\boldsymbol{\eta}\right\|+ h^{2} |||\boldsymbol{\eta}|||_{2}\right),
\end{equation}
and
\begin{equation}\label{seminormestimate}
|||\boldsymbol{\eta}|||_{*} \leq C \left(\left\|\nabla\left(\boldsymbol{u}-I_{h}\boldsymbol{u}\right)\right\|+\left\|\nabla\boldsymbol{\eta}\right\|+h |||\boldsymbol{\eta}|||_{2}\right).
\end{equation}
\end{lemma}
\begin{proof}
Note that
\begin{equation}\label{piheta}
\begin{aligned}
\Pi_{h}\boldsymbol{\eta}=\boldsymbol{u}-\Pi_{h}\Pi_{h}^{S}\boldsymbol{u}&=\boldsymbol{u}-I_{h}\Pi_{h}^{S}\boldsymbol{u}-\Pi_{h}^{RT}\left(1-I_{h}\right)\Pi_{h}^{S}\boldsymbol{u}\\
&=\boldsymbol{\eta}+\left[\left(1-I_{h}\right)\Pi_{h}^{S}\boldsymbol{u}-\Pi_{h}^{RT}\left(1-I_{h}\right)\Pi_{h}^{S}\boldsymbol{u}\right],
\end{aligned}
\end{equation}
\begin{equation}\label{piheta1}
\Pi_{h}^{1}\boldsymbol{\eta}=\boldsymbol{u}-I_{h}\Pi_{h}^{S}\boldsymbol{u}=\boldsymbol{\eta}+\left(1-I_{h}\right)\Pi_{h}^{S}\boldsymbol{u},
\end{equation}
and
\begin{equation}\label{pihetaR}
\Pi_{h}^{R}\boldsymbol{\eta}=-\Pi_{h}^{R}\Pi_{h}^{S}\boldsymbol{u}=-\Pi_{h}^{RT}\left(1-I_{h}\right)\Pi_{h}^{S}\boldsymbol{u}.
\end{equation}
A common term in the right-hand sides of \crefrange{piheta}{pihetaR} is
\begin{equation}\label{mainerrorsplit}
\left(1-I_{h}\right)\Pi_{h}^{S}\boldsymbol{u}=\left(\boldsymbol{u}-I_{h}\boldsymbol{u}\right)+I_{h}\boldsymbol{\eta}.
\end{equation}
Note that \cref{Ihestimate1} guarantees with the triangle inequality
\begin{equation}\label{Ihbound1}
\left\|I_{h}\boldsymbol{v}\right\|_{K}\leq \left\|\boldsymbol{v}\right\|_{K}+C h_{K}^{2} \left|\boldsymbol{v}\right|_{2,K},
\end{equation}
and
\begin{equation}\label{Ihbound2}
\left\|\nabla I_{h}\boldsymbol{v}\right\|_{K}\leq \left\|\nabla \boldsymbol{v}\right\|_{K}+C h_{K} \left|\boldsymbol{v}\right|_{2,K},
\end{equation}
for all $\boldsymbol{v}\in \boldsymbol{C}^{0}\left(\bar{\Omega}\right)\cap \boldsymbol{H}^{2}\left(\mathcal{T}_{h}\right), K\in\mathcal{T}_{h}$.
Substituting \cref{Ihbound1} and \cref{Ihbound2} into \cref{mainerrorsplit} gives
\begin{equation}\label{1minusIhestimate}
\left|\left(1-I_{h}\right)\Pi_{h}^{S}\boldsymbol{u}\right|_{m,K}\leq \left|\boldsymbol{u}-I_{h}\boldsymbol{u}\right|_{m,K}+\left|\boldsymbol{\eta}\right|_{m,K}+C h_{K}^{2-m} \left|\boldsymbol{\eta}\right|_{2,K} \text{ for } m=0,1.
\end{equation}
Substituting \cref{1minusIhestimate} into \crefrange{piheta}{pihetaR} and applying \cref{pihrtestimate2} for \cref{piheta,pihetaR} provide \crefrange{piestimate}{piestimater}.

To estimate $|||\boldsymbol{\eta}|||_{*}$, with \cref{pihetaR}, \cref{pihrtestimate2}, \cref{Ihestimate2} and the inverse inequality one obtains
\begin{displaymath}
\begin{aligned}
\left\|\Pi_{h}^{R}\boldsymbol{\eta}\right\|_{K}\leq C \left\|\left(1-I_{h}\right)\Pi_{h}^{S}\boldsymbol{u}\right\|_{K}&\leq C h_{K}^{2}  \left|\left(1-I_{h}\right)\Pi_{h}^{S}\boldsymbol{u}\right|_{2,K}\\
&\leq C h_{K}  \left\|\nabla\left(1-I_{h}\right)\Pi_{h}^{S}\boldsymbol{u}\right\|_{K}.
\end{aligned}
\end{displaymath}
The above estimate, together with \cref{seminorm,1minusIhestimate}, implies that
\begin{displaymath}
|||\boldsymbol{\eta}|||_{*}\leq C \left\|\nabla\left(1-I_{h}\right)\Pi_{h}^{S}\boldsymbol{u}\right\|\leq C \left(\left\|\nabla\left(\boldsymbol{u}-I_{h}\boldsymbol{u}\right)\right\|+\left\|\nabla\boldsymbol{\eta}\right\|+h |||\boldsymbol{\eta}|||_{2}\right).
\end{displaymath}
This completes the proof.
\end{proof}
Finally, based on the results in \cref{mainresulttheorem}, \cref{consistencylemma} and \cref{piestimatelemma}, as well as the approximation properties of $I_{h}$ and $\Pi_{h}^{S}$ \cite{brenner_mathematical_2008,girault_finite_1986}, we get the convergence rates of the kinetic and dissipation energy errors of $\boldsymbol{u}_{h}$ (or $\Pi_{h}\boldsymbol{u}_{h}$) for \cref{EMAPRsemidiscreteform}, with the elements and reconstruction operators in this section.
\begin{corollary}
Let $\left(\boldsymbol{u},p\right)$ be the solution of \cref{weakcontinuousform} and $\left(\boldsymbol{u}_{h},p_{h}\right)$ be the solution of \cref{EMAPRsemidiscreteform}, with the elements and reconstruction operators used in \Cref{sec:4}. Suppose $\boldsymbol{u}\in L^{\infty}\left(J;\boldsymbol{H}^{k+1}\left(\Omega\right)\right)\cap L^{2}\left(J;\boldsymbol{W}^{1,\infty}\left(\Omega\right)\right)$, $\boldsymbol{u}_{t}\in L^{2}\left(J;\boldsymbol{H}^{k+1}\left(\Omega\right)\right)$. Under \cref{thirdassumption} and the assumption that $\boldsymbol{u}_{h}^{0}=\Pi_{h}^{S}\boldsymbol{u}^{0}$, the following estimate holds:
\begin{equation}\label{maincorollary}
\left\|\boldsymbol{u}-\boldsymbol{u}_{h}\right\|_{L^{\infty}\left(J; \boldsymbol{L}^{2}\left(\Omega\right)\right)}+\nu^{\frac{1}{2}}\left\|\nabla\left(\boldsymbol{u}-\boldsymbol{u}_{h}\right)\right\|_{L^{2}\left(J; \boldsymbol{L}^{2}\left(\Omega\right)\right)}
\leq B\left(\boldsymbol{u},T\right)h^{k},
\end{equation}
where
\begin{displaymath}
\begin{aligned}
B&\left(\boldsymbol{u},T\right)=C \bigg\{h\left|\boldsymbol{u}\right|_{L^{\infty}\left(J; \boldsymbol{H}^{k+1}\left(\Omega\right)\right)}+\left|\boldsymbol{u}\right|_{L^{2}\left(J;\boldsymbol{H}^{k+1}\left(\Omega\right)\right)}
+e^{\frac{1}{2}G\left(\boldsymbol{u},T\right)}\Big[\\& h\left|\boldsymbol{u}_{t}\right|_{L^{2}\left(J;\boldsymbol{H}^{k+1}\left(\Omega\right)\right)}+
\left|\boldsymbol{u}\right|_{L^{2}\left(J;\boldsymbol{H}^{k+1}\left(\Omega\right)\right)}+\left\|\boldsymbol{u}\right\|_{L^{2}
\left(J;\boldsymbol{W}^{1,\infty}\left(\Omega\right)\right)}\left|\boldsymbol{u}\right|_{L^{4}\left(J;\boldsymbol{H}^{k+1}\left(\Omega\right)\right)}\Big]\bigg\},
\end{aligned}
\end{displaymath}
with $C$ independent of $p$, $h$ and inverse powers of $\nu$.
\end{corollary}
\section{Numerical experiments}
\label{sec:5}
\subsection{Example 1: convergence test and pressure-robustness test}
\label{sec:51}
For the first example we consider the potential flow in \cite[Example 6]{linke2016}. On $\Omega=\left(0,1\right)^{2}$ the velocity is prescribed as $\boldsymbol{u}=\min\left\{t,1\right\}\nabla\chi$ with $\chi=x^{3}y-y^{3}x$. We set $\boldsymbol{f}=0$ such that the pressure gradient exactly balances the gradient filed produced by the velocity terms. Due to the quadratic convective term, the pressure is much more complicated than the velocity \cite{gauger_high-order_2019}. The pressure-robustness will paly a key role on the accuracy of the simulations of this problem. We use this example to show the convergence rates and the pressure-robustness of our method.

We consider the case of $\nu=5\times 10^{-4}$ and apply the first order Bernardi-Raugel element and the second order element mentioned in \cref{sec:4}, $P_{2}^{bubble}/P_{1}^{disc}$. For the time-stepping, the BDF2 scheme is used. For convergence test, here we only consider the spatial effects and neglect the effects of the time discretizations by choosing a small time step: $\Delta t=0.001$ and $T=0.1$. A non-uniform initial mesh is used, which consists of 132 triangles and produces total 506 DOFs for Bernardi-Raugel element and 1246 DOFs for $P_{2}^{bubble}/P_{1}^{disc}$. The pressure-robustness test is performed on the double refinement of the initial mesh, with the time step $\Delta t=0.01$ and $T=2$.
We also give some results from the classical scheme and the pressure-robust reconstruction scheme in \cite{linke2016} with the convective form for the nonlinear term.
In each time step, we solve a linear problem by replacing the advective velocity in trilinear forms with some appropriate extrapolation of the previous step velocities.

Finally, to show the pressure-robustness of our methods further, we also compute the problem with $\boldsymbol{f}=100\nabla\chi$, which only change the pressure in the continuous problem and the discrete pressure for pressure-robust methods. Some results are shown in \cref{tab:1,tab:2,tab:3,tab:4,tab:5,tab:6,tab:7}. For pressure-robust tests, we only show the results with $\alpha=0$. The results below, especially the results in \cref{tab:7}, demonstrate that our methods are robust with respect to the continuous pressure.

For second order tests, we find that our methods give a little worse result than the pressure-robust reconstruction methods in \cite{linke2016}. This is reasonable since we alter the discretization of the convective form. Although our methods are EMA-conserving, these advantages is not easy to shown in a potential flow unless we use a much smaller viscosity.
\begin{table}
\centering
\setlength\tabcolsep{3pt}
\caption{Example 1. Errors by the Bernardi-Raugel element with $\alpha=0, T=0.1$.}
\begin{tabular}{ccccccccc}
\hline level & $\left\|\boldsymbol{u}-\boldsymbol{u}_{h}\right\|$ & eoc & $\left\|\boldsymbol{u}-\Pi_{h}\boldsymbol{u}_{h}\right\|$ & eoc & $\left\|\nabla\left(\boldsymbol{u}-\boldsymbol{u}_{h}\right)\right\|$ & eoc & $\left\|p-p_{h}\right\|$ & eoc \\
\hline 0 & $4.3762  ${e}-$4 $ & $-$ & $5.5822${e}-$4$ & $-$ & $2.7562${e}-$2 $ & $-$ & $2.2363  ${e}-$2 $ & $-$ \\
\hline 1 & $1.0781 ${e}-$4 $ & $2.02 $ & $1.4006 ${e}-$4 $ & $1.99 $ & $1.3462${e}-$2 $ & $1.03$ & $1.1242 ${e}-$2 $ & $0.99$ \\
\hline 2 & $2.6214${e}-$5 $ & $2.04 $ & $3.4968${e}-$5$ & $2.00 $ & $6.5417 ${e}-$3 $ & $1.04$ & $5.6197 ${e}-$3 $ & $1.00$ \\
\hline 3 & $6.4910${e}-$6 $ & $2.01 $ & $8.7537 ${e}-$6$ & $1.99 $ & $3.2321 ${e}-$3 $ & $1.01$ & $2.8093 ${e}-$3$ & $1.00$ \\
\hline
\end{tabular}
\label{tab:1}
\end{table}
\begin{table}
\centering
\setlength\tabcolsep{3pt}
\caption{Example 1. Errors by $P_{2}^{bubble}/P_{1}^{disc}$ with $\alpha=0,T=0.1$.}
\begin{tabular}{ccccccccc}
\hline level & $\left\|\boldsymbol{u}-\boldsymbol{u}_{h}\right\|$ & eoc & $\left\|\boldsymbol{u}-\Pi_{h}\boldsymbol{u}_{h}\right\|$ & eoc & $\left\|\nabla\left(\boldsymbol{u}-\boldsymbol{u}_{h}\right)\right\|$ & eoc & $\left\|p-p_{h}\right\|$ & eoc \\
\hline 0 & $1.1624 ${e}-$5 $ & $-$ & $1.4688${e}-$5$ & $-$ & $9.4495${e}-$4 $ & $-$ & $1.2894 ${e}-$3 $ & $-$ \\
\hline 1 & $1.4175${e}-$6 $ & $3.03 $ & $1.9025 ${e}-$6$ & $2.94 $ & $2.2039${e}-$4 $ & $2.10$ & $3.2351${e}-$4$ & $1.99$ \\
\hline 2 & $1.8333 ${e}-$7 $ & $2.95 $ & $2.4828 ${e}-$7 $ & $2.93$ & $5.4150${e}-$5$ & $2.02$ & $8.0835 ${e}-$5 $ & $2.00$ \\
\hline 3 & $2.3777 ${e}-$8 $ & $2.94 $ & $3.2014 ${e}-$8 $ & $2.95$ & $1.3519 ${e}-$5$ & $2.00$ & $2.0233${e}-$5 $ & $1.99$ \\
\hline
\end{tabular}
\label{tab:2}
\end{table}
\begin{table}
\centering
\setlength\tabcolsep{3pt}
\caption{Example 1. Errors by the Bernardi-Raugel element with $\alpha=1,T=0.1$.}
\begin{tabular}{ccccccccc}
\hline level & $\left\|\boldsymbol{u}-\boldsymbol{u}_{h}\right\|$ & eoc & $\left\|\boldsymbol{u}-\Pi_{h}\boldsymbol{u}_{h}\right\|$ & eoc & $\left\|\nabla\left(\boldsymbol{u}-\boldsymbol{u}_{h}\right)\right\|$ & eoc & $\left\|p-p_{h}\right\|$ & eoc \\
\hline 0 & $4.7428${e}-$4 $ & $-$ & $5.7279${e}-$4$ & $-$ & $2.8068${e}-$2 $ & $-$ & $2.2364 ${e}-$2 $ & $-$ \\
\hline 1 & $1.0993 ${e}-$4 $ & $2.10 $ & $1.4107 ${e}-$4 $ & $2.02 $ & $1.3287${e}-$2 $ & $1.07$ & $1.1242${e}-$2 $ & $0.99$ \\
\hline 2 & $2.6469${e}-$5 $ & $2.05 $ & $3.5056${e}-$5$ & $2.00 $ & $6.5124 ${e}-$3 $ & $1.02$ & $5.6197 ${e}-$3 $ & $1.00$ \\
\hline 3 & $6.5237${e}-$6 $ & $2.02 $ & $8.7597${e}-$6$ & $2.00 $ & $3.2315 ${e}-$3 $ & $1.01$ & $2.8093${e}-$3$ & $1.00$ \\
\hline
\end{tabular}
\label{tab:3}
\end{table}
\begin{table}
\centering
\setlength\tabcolsep{3pt}
\caption{Example 1. Errors by $P_{2}^{bubble}/P_{1}^{disc}$ with $\alpha=1,T=0.1$.}
\begin{tabular}{ccccccccc}
\hline level & $\left\|\boldsymbol{u}-\boldsymbol{u}_{h}\right\|$ & eoc & $\left\|\boldsymbol{u}-\Pi_{h}\boldsymbol{u}_{h}\right\|$ & eoc & $\left\|\nabla\left(\boldsymbol{u}-\boldsymbol{u}_{h}\right)\right\|$ & eoc & $\left\|p-p_{h}\right\|$ & eoc \\
\hline 0 & $1.2393${e}-$5 $ & $-$ & $1.6257${e}-$5$ & $-$ & $9.1896${e}-$4 $ & $-$ & $1.2894 ${e}-$3 $ & $-$ \\
\hline 1 & $1.4720${e}-$6 $ & $3.07 $ & $2.0064 ${e}-$6$ & $3.01 $ & $2.1807${e}-$4 $ & $2.07$ & $3.2351${e}-$4$ & $1.99$ \\
\hline 2 & $1.8534 ${e}-$7 $ & $2.98 $ & $2.5244 ${e}-$7 $ & $2.99 $ & $5.4057${e}-$5$ & $2.01$ & $8.0835 ${e}-$5 $ & $2.00$ \\
\hline 3 & $2.3864 ${e}-$8 $ & $2.95$ & $3.2125${e}-$8 $ & $2.97 $ & $1.3523 ${e}-$5$ & $1.99$ & $2.0233${e}-$5 $ & $1.99$ \\
\hline
\end{tabular}
\label{tab:4}
\end{table}
\begin{table}
\centering
\caption{Example 1. Errors by the Bernardi-Raugel element with $\alpha=0$ on mesh level 2 (classical methods/EMAPR/pressure-robust reconstructions).}
\begin{tabular}{cccc}
\hline t & $\left\|\boldsymbol{u}-\boldsymbol{u}_{h}\right\|$ &  $\left\|\nabla\left(\boldsymbol{u}-\boldsymbol{u}_{h}\right)\right\|$  & $\left\|P_{h}p-p_{h}\right\|$ \\
\hline 0.5 & $1.84${e}-$2 $/$1.61${e}-$4$/$1.62${e}-$4 $ & $3.66$/$3.72${e}-$2 $/$3.86 ${e}-$2 $
  & $1.63${e}-$2 $/$5.02${e}-$5 $/$4.70 ${e}-$5 $ \\
\hline 1 & $4.35 ${e}-$2 $/$4.47 ${e}-$4 $/$4.71 ${e}-$4$ & $8.47$/$8.73 ${e}-$2 $/$9.83${e}-$2 $
  & $5.59${e}-$2 $/$2.46 ${e}-$4$/$2.29${e}-$4 $ \\
\hline 1.5 & $5.50 ${e}-$2 $/$5.56 ${e}-$4$/$6.07 ${e}-$4$ & $9.24$/$8.89 ${e}-$2 $/$0.10$
  & $4.02${e}-$2 $/$2.62 ${e}-$4$/$2.37${e}-$4 $ \\
\hline 2.0 & $5.53 ${e}-$2 $/$5.98${e}-$4 $/$7.37 ${e}-$4 $ & $9.24$/$8.90 ${e}-$2 $/$0.10$
  & $4.09${e}-$2 $/$2.84${e}-$4$/$2.68${e}-$4$ \\
\hline
\end{tabular}
\label{tab:5}
\end{table}
\begin{table}
\centering
\caption{Example 1. Errors by $P_{2}^{bubble}/P_{1}^{disc}$ with $\alpha=0$ on mesh level 2 (classical methods/EMAPR/pressure-robust reconstructions).}
\begin{tabular}{cccc}
\hline t & $\left\|\boldsymbol{u}-\boldsymbol{u}_{h}\right\|$ &  $\left\|\nabla\left(\boldsymbol{u}-\boldsymbol{u}_{h}\right)\right\|$  & $\left\|P_{h}p-p_{h}\right\|$ \\
\hline 0.5 & $4.85${e}-$5 $/$1.61${e}-$6 $/$1.20 ${e}-$6 $ & $1.21${e}-$2 $/$4.55 ${e}-$4 $/$3.36 ${e}-$4$
  & $6.94${e}-$5 $/$1.09${e}-$6$/$5.80 ${e}-$7 $ \\
\hline 1 & $9.81${e}-$5 $/$4.95${e}-$6 $/$3.40${e}-$6 $ & $2.58 ${e}-$2 $/$1.34 ${e}-$3 $/$8.69${e}-$4 $
  & $1.50 ${e}-$4$/$5.14${e}-$6 $/$2.54 ${e}-$6 $ \\
\hline 1.5 & $1.15${e}-$4 $/$5.75 ${e}-$6$/$3.80 ${e}-$6$ & $2.65 ${e}-$2 $/$1.36 ${e}-$3 $/$8.95 ${e}-$4 $
  & $1.44 ${e}-$4 $/$5.29${e}-$6 $/$2.64${e}-$6$ \\
\hline 2.0 & $1.17 ${e}-$4 $/$6.07 ${e}-$6 $/$3.94${e}-$6$ & $2.66${e}-$2 $/$1.37 ${e}-$3 $/$8.98${e}-$4 $
  & $1.43${e}-$4 $/$5.47${e}-$6 $/$2.73 ${e}-$6$ \\
\hline
\end{tabular}
\label{tab:6}
\end{table}
\begin{table}
\centering
\caption{Example 1. Errors with $\alpha=0$ and $\boldsymbol{f}=100\nabla\chi$ on mesh level 2 (classical methods/EMAPR).}
\setlength\tabcolsep{3pt}
\begin{tabular}{cccccc}
\hline & \multicolumn{2}{c}{Bernardi-Raugel}  & &  \multicolumn{2}{c}{$P_{2}^{bubble}/P_{1}^{disc}$}  \\
\cline { 2 - 3 } \cline { 5 - 6 }$t$ &$\|\boldsymbol{u}-\boldsymbol{u}_{h}\|$ &$\|\nabla(\boldsymbol{u}-\boldsymbol{u}_{h})\|$ & &$\|\boldsymbol{u}-\boldsymbol{u}_{h}\|$ &$\|\nabla(\boldsymbol{u}-\boldsymbol{u}_{h})\|$ \\
\hline$0.5$    &$1.51$/$1.61${e}-$4$ &$265.73$/$3.72${e}-$2$ & &$4.80${e}-$3$/$1.61${e}-$6$ &$1.10$/$4.55${e}-$4$ \\
$1$            &$1.62$/$4.47${e}-$4$ &$285.33$/$8.73${e}-$2$ & &$3.74${e}-$3$/$4.95${e}-$6$ &$0.90$/$1.34${e}-$3$ \\
$1.5$          &$1.74$/$5.56${e}-$4$ &$284.10$/$8.89${e}-$2$ & &$3.78${e}-$3$/$5.75${e}-$6$ &$0.90$/$1.36${e}-$3$ \\
$2$            &$1.78$/$5.98${e}-$4$ &$280.08$/$8.90${e}-$2$ & &$3.86${e}-$3$/$6.07${e}-$6$ &$0.90$/$1.37${e}-$3$ \\
\hline
\end{tabular}
\label{tab:7}
\end{table}
\subsection{Example 2: EMA-conserving test: the Gresho problem}
\label{sec:52}
In the second example we consider the Gresho problem \cite{Rebholz2017,EMAC2019,gauger_high-order_2019}, which is a benchmark to test the EMA-conserving properties of a method. With $\boldsymbol{f}=\boldsymbol{0}$ and $\nu=0$, the exact solutions on $\Omega=\left(-0.5,0.5\right)^{2}$ are set as
$$
\begin{array}{c}
r \leq 0.2:\left\{\begin{array}{l}
\boldsymbol{u}=\left(\begin{array}{c}
-5 y \\
5 x
\end{array}\right), \\
p=12.5 r^{2}+\gamma
\end{array}\right.
0.2 \leq r \leq 0.4:\left\{\begin{array}{l}
\boldsymbol{u}=\left(\begin{array}{c}
-\frac{2 y}{r}+5 y \\
\frac{2 x}{r}-5 x
\end{array}\right), \\
p=12.5 r^{2}-20 r+4 \log \left(r\right)+\beta
\end{array}\right.
\end{array}
$$
and all vanish for $r>0.4$, where $r=\sqrt{x^{2}+y^{2}}$ and
$$
\beta=\left(-12.5\right)\left(0.4\right)^{2}+20\left(0.4\right)^{2}-4 \log \left(0.4\right), \gamma=\beta-20\left(0.2\right)+4 \log \left(0.2\right).
$$

We strongly enforce the no-penetration boundary condition in computations. And set $\alpha=0$ and $\alpha=1$ for Bernardi-Raugel element and $P_{2}^{bubble}/P_{1}^{disc}$, respectively. We find that for this problem the gradient of the velocity solution might be very large for higher order elements ($k\geq2$) if $\alpha=0$. This is the reason for the choice of $\alpha$. The Bernardi-Raugel element is tested on a uniform $48\times48$ triangular mesh and the $P_{2}^{bubble}/P_{1}^{disc}$ is tested on a non-uniform mesh with $h=1/25$. To highlight the conservative properties, we apply the Crank-Nicolson scheme for time discretizations with $\Delta t=0.01$ and $T=10$.

To make a comparison, we also compute the results from two classes of conservative methods: one is the EMAC formulation \cite{Rebholz2017,Rebholz2020} with the same elements and meshes as our methods, the other is the classical convective formulation but with the exactly divergence-free elements. For the first order divergence-free elements, we choose the element proposed by Guzm\'{a}n and Neilan in \cite{Neilan2018}, which is performed on the same mesh as the Bernardi-Raugel element. Note that the Guzm\'{a}n-Neilan element has the same DOFs as the Bernardi-Raugel element on a given mesh. The Guzm\'{a}n-Neilan element consists of linear pieceewise polynomials and some modified Bernardi-Raugel bubbles which are constructed on the barycentric refinement of each triangle. For the second order divergence-free elements, we choose the well-known Scott-Vogelius (SV2) element, $P_{2}/P_{1}^{disc}$ \cite{Arnold1992,john_divergence_2017}, which is run on the barycentric refinement of the mesh for $P_{2}^{bubble}/P_{1}^{disc}$ to guarantee the stability. For all methods we solve a nonlinear system in each step by Newton iterations or Picard iterations with a tolerance of $10^{-6}$ for $H^{1}$ norm.

Some results are shown in \cref{Greshofig1} and \cref{Greshofig2}, where the ``momentum" denotes the sum of all the components of the linear momentum. For EMAPR methods, all the quantities are computed by $\Pi_{h}\boldsymbol{u}_{h}$. For first order approximations, the pressure is approximated by piecewise constant. In this time, the effect of the lack of pressure-robustness for the EMAC formulation is obvious. For higher order approximations, all the methods give very similar results, since the continuous pressure is not very complicated (the maximum power is 2) and all the methods are EMA-conserving.
\begin{figure}[htbp]
\centering
\includegraphics[width=0.48\textwidth,height=4.2cm]{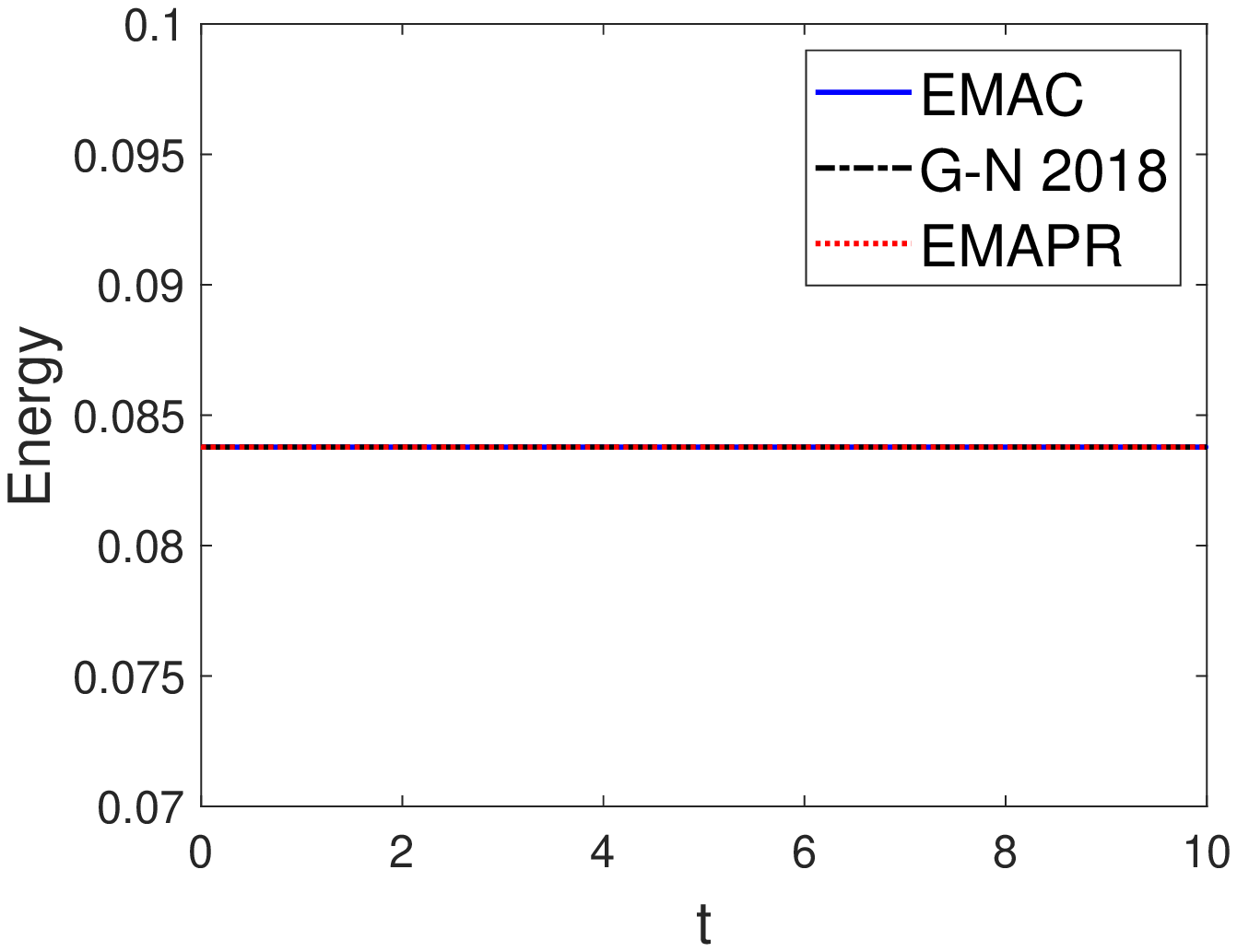}
\includegraphics[width=0.48\textwidth,height=4.2cm]{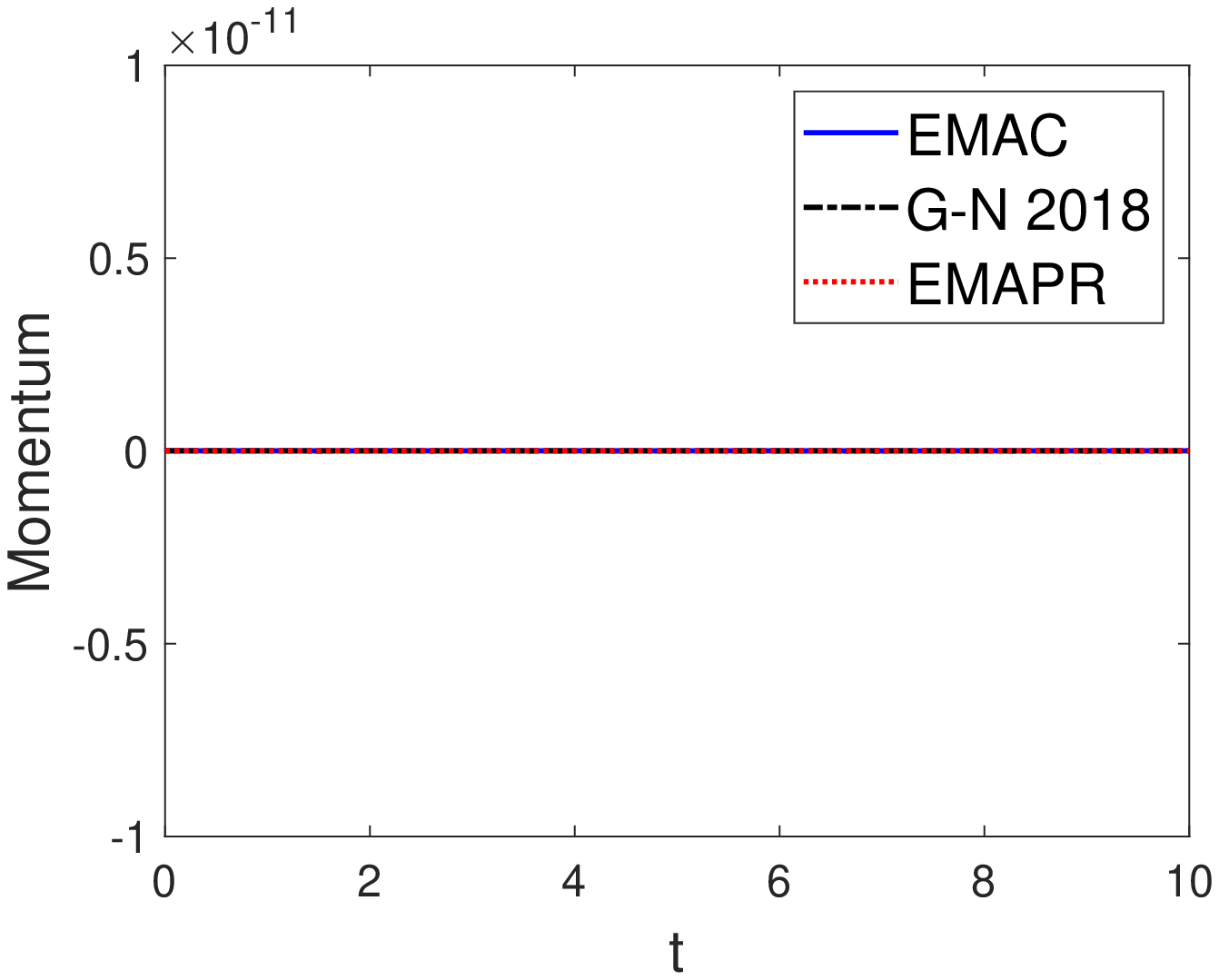}
\includegraphics[width=0.48\textwidth,height=4.2cm]{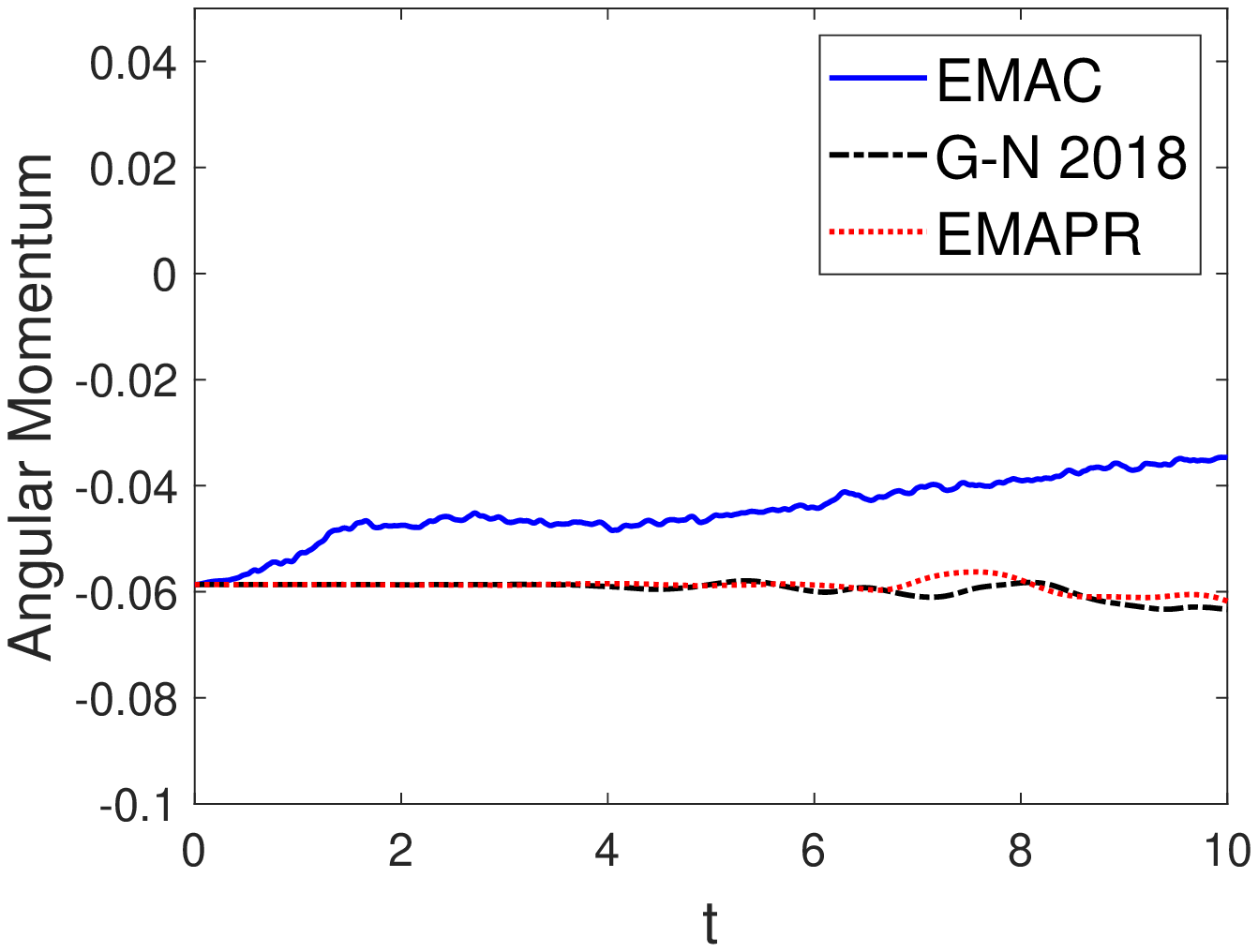}
\includegraphics[width=0.48\textwidth,height=4.2cm]{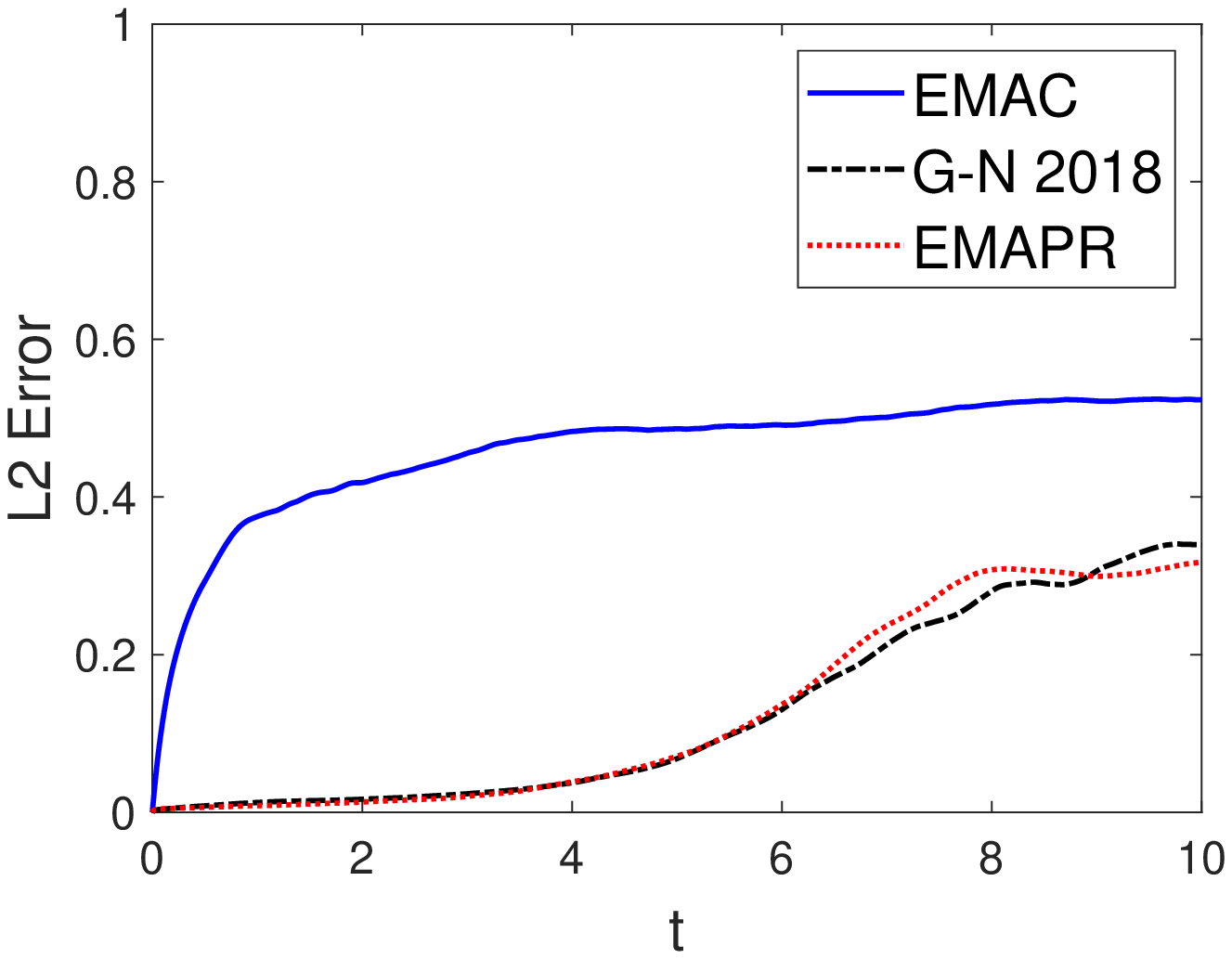}
\caption{Example 2. Plots of kinetic energy, momentum, angular momentum and L2 errors by the Bernardi-Raugel element or Guzm\'{a}n-Neilan element (G-N 2018) versus time.}
\label{Greshofig1}
\end{figure}
\begin{figure}[htbp]
\centering
\includegraphics[width=0.48\textwidth,height=4.2cm]{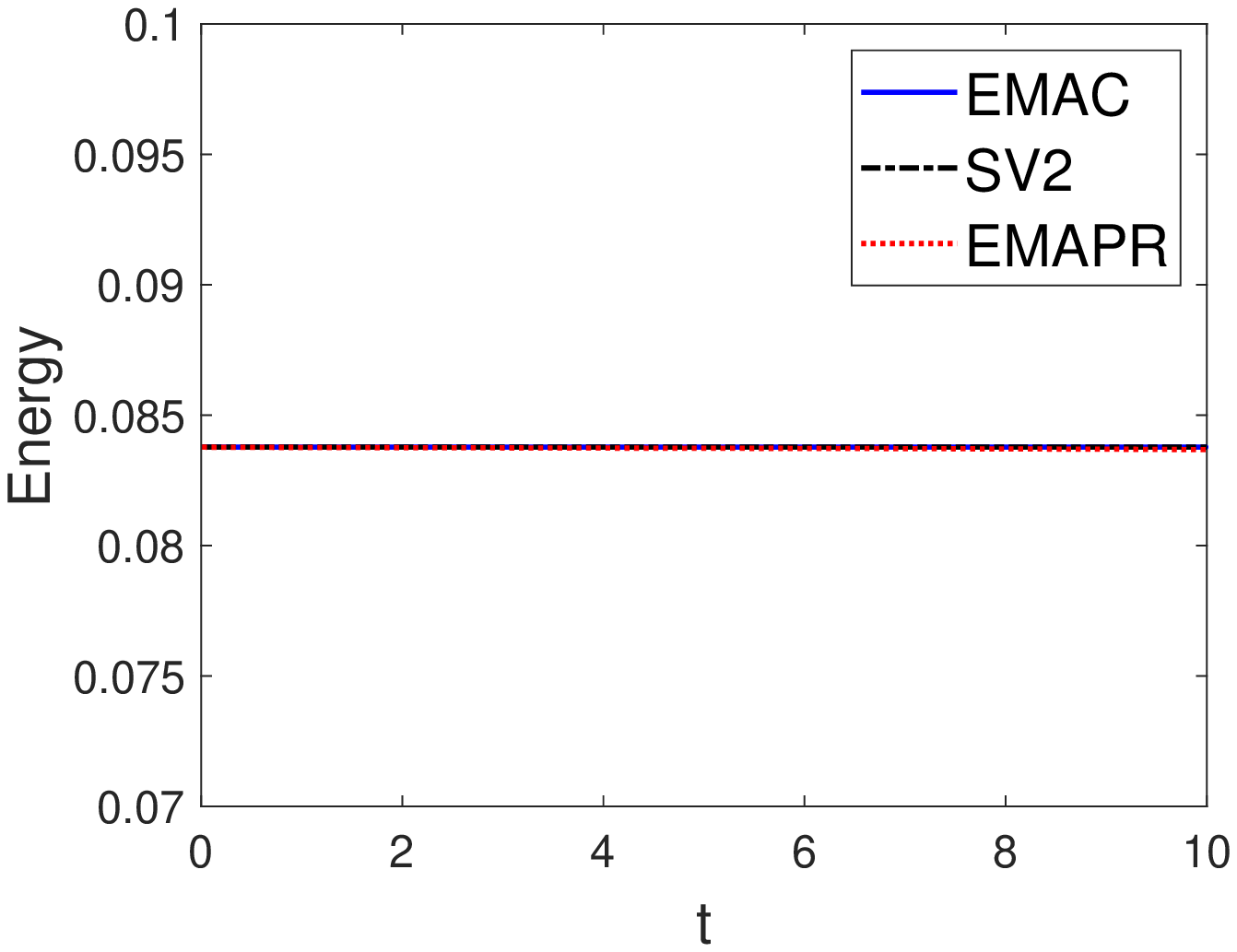}
\includegraphics[width=0.48\textwidth,height=4.2cm]{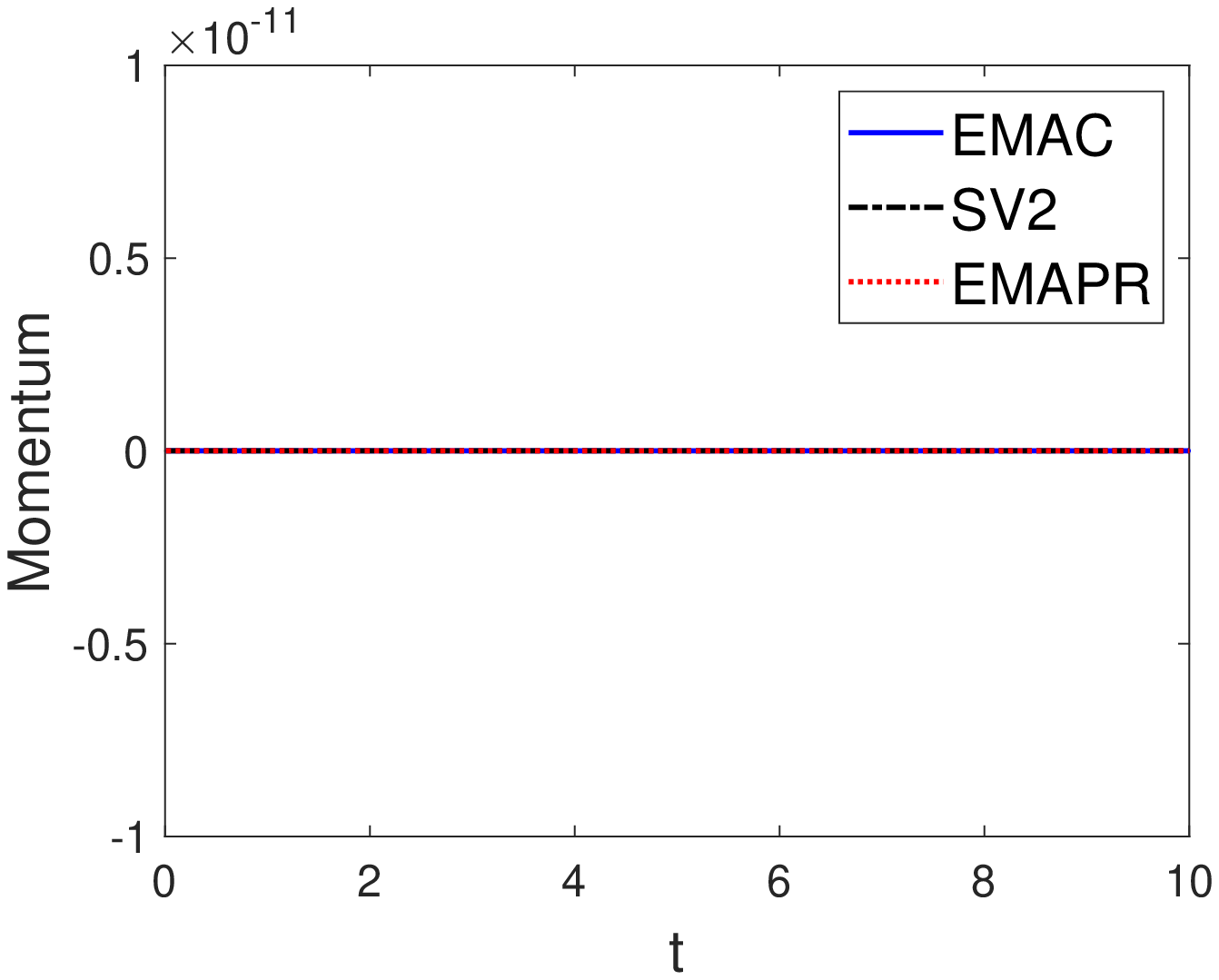}
\includegraphics[width=0.48\textwidth,height=4.2cm]{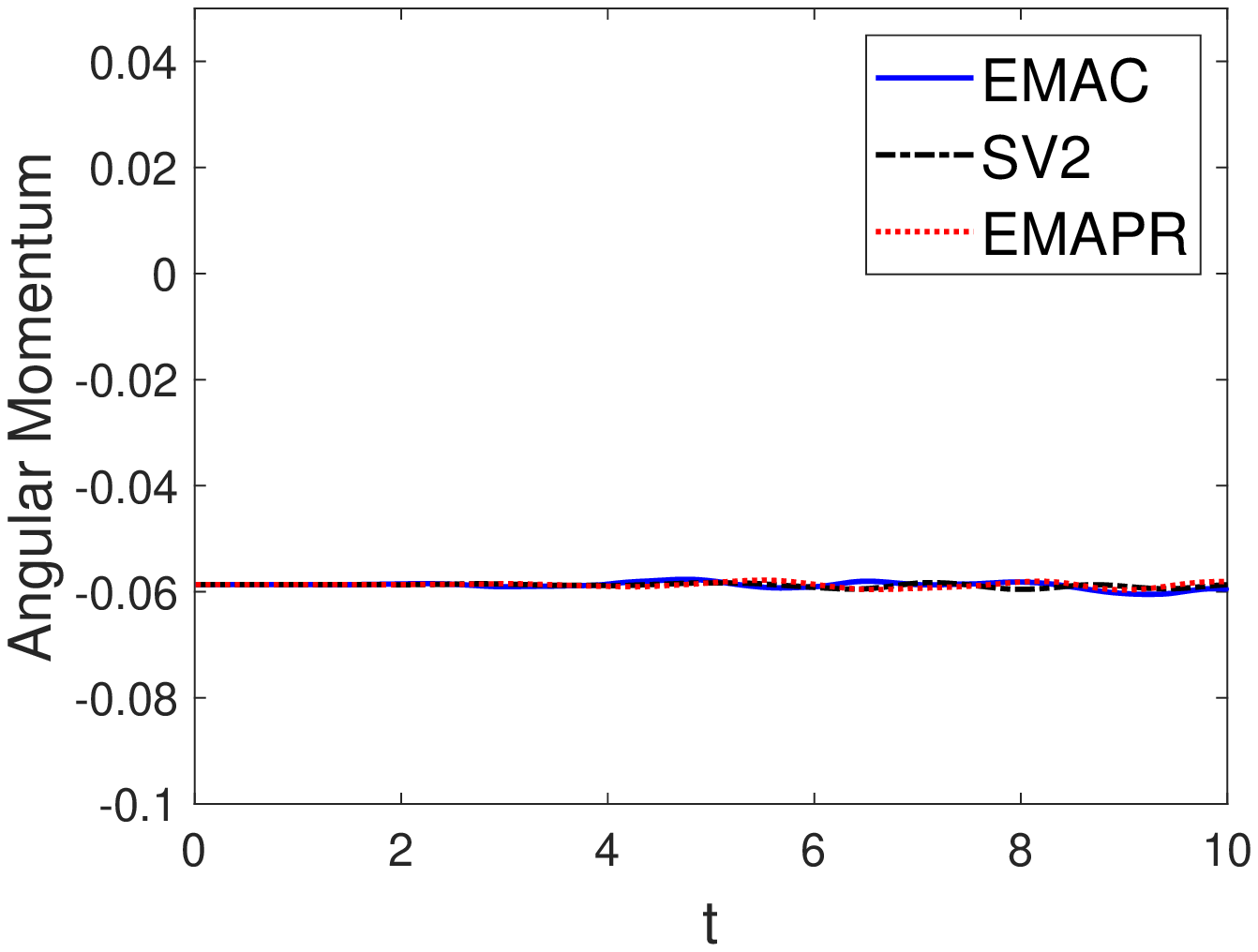}
\includegraphics[width=0.48\textwidth,height=4.2cm]{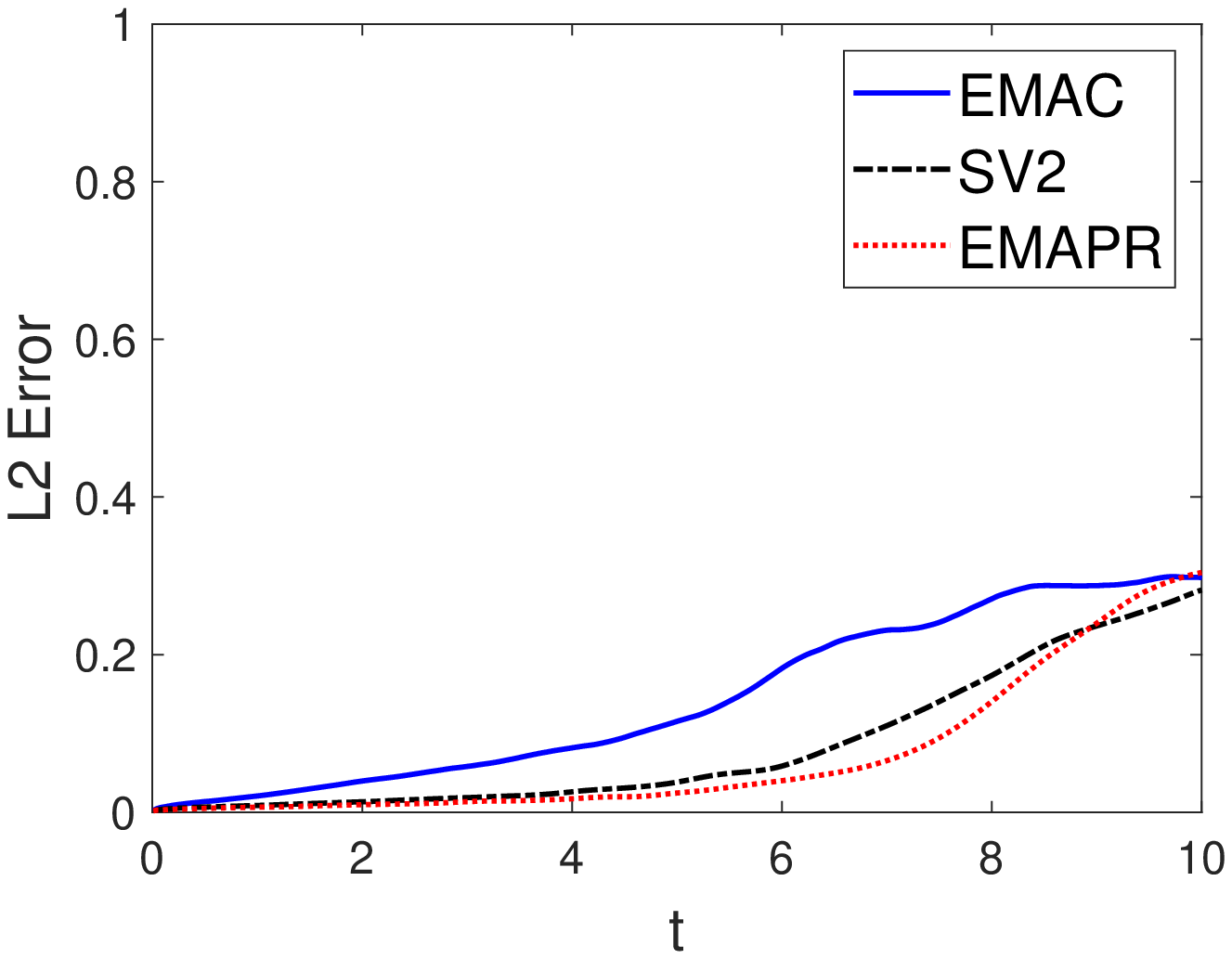}
\caption{Example 2. Plots of kinetic energy, momentum, angular momentum and L2 errors by $P_{2}^{bubble}/P_{1}^{disc}$ or SV2 versus time.}
\label{Greshofig2}
\end{figure}
\subsection{Example 3: $Re$-semi-robustness test: the lattice vortex problem}
\label{sec:53}
In the final example, we consider the lattice vortex problem \cite{Rebholz2020,schroeder_towards_2018} on $\Omega=\left(0,1\right)^{2}$, which is a benchmark to test the exponential growth rates (with respect to time) of the errors. In \cref{sec:4} we have shown that the Gronwall constant is independent of $\nu$. The exact velocity is set as $\boldsymbol{u}\left(t,\boldsymbol{x}\right)=\boldsymbol{u}^{0}\left(\boldsymbol{x}\right)\exp\left({-8\pi^{2}\nu t}\right)$ with $\boldsymbol{u}^{0}\left(\boldsymbol{x}\right)=\left(\sin\left(2\pi x\right)\sin\left(2\pi y\right),\cos\left(2\pi x\right)\cos\left(2\pi y\right)\right)^{\text{T}}$. With an appropriate $p$, $\boldsymbol{u}$ fulfills an exact unsteady NSE with $\boldsymbol{f}=\boldsymbol{0}$. To test the $Re$-semi-robustness, we choose a small $\nu$ and large $T$: $\nu=1\times10^{-5}$ and $T=10$.

The methods (or elements) used in this example are the same as \cref{sec:52}, except replacing the EMAC scheme with the classical skew-symmetric scheme (SKEW), which has been shown not to be $Re$-semi-robust with non-divergence-free elements \cite{Rebholz2020,schroeder_towards_2018}. All the first order methods are run on the uniform $64\times64$ triangular mesh. The $P_{2}^{bubble}/P_{1}^{disc}$ are tested on a non-uniform mesh with the size $h=0.03$, and the SV2 element is performed on the barycentric refinement of the same mesh. For our methods, we give the results for both $\alpha=0$ and $\alpha=1$. Note that from the theoretical analysis the value of $\alpha$ has an effect on the property of $Re$-semi-robustness. For the time discretizations, we use the Crank-Nicolson scheme with $\Delta t=0.001$. We linearize all the methods by replacing the first velocity in trilinear forms with some appropriate extrapolation of the previous step velocities.

Some results are shown in \cref{latticefig1,latticefig2}. One could find that the value of $\alpha$ does have an effect on the growth speed of the errors. For first order approximations, the methods (elements) proposed by Guzm\'{a}n and Neilan in \cite{Neilan2018} give the best results, and our method with $\alpha=1$ gives very close performance in the final time. For second order approximations, our method admits the best performance. Except the SKEW formulation on non-divergence-free elements, all the methods below show a slower growth speed of the errors.
\begin{figure}[htbp]
\centering
\includegraphics[width=0.48\textwidth,height=4.2cm]{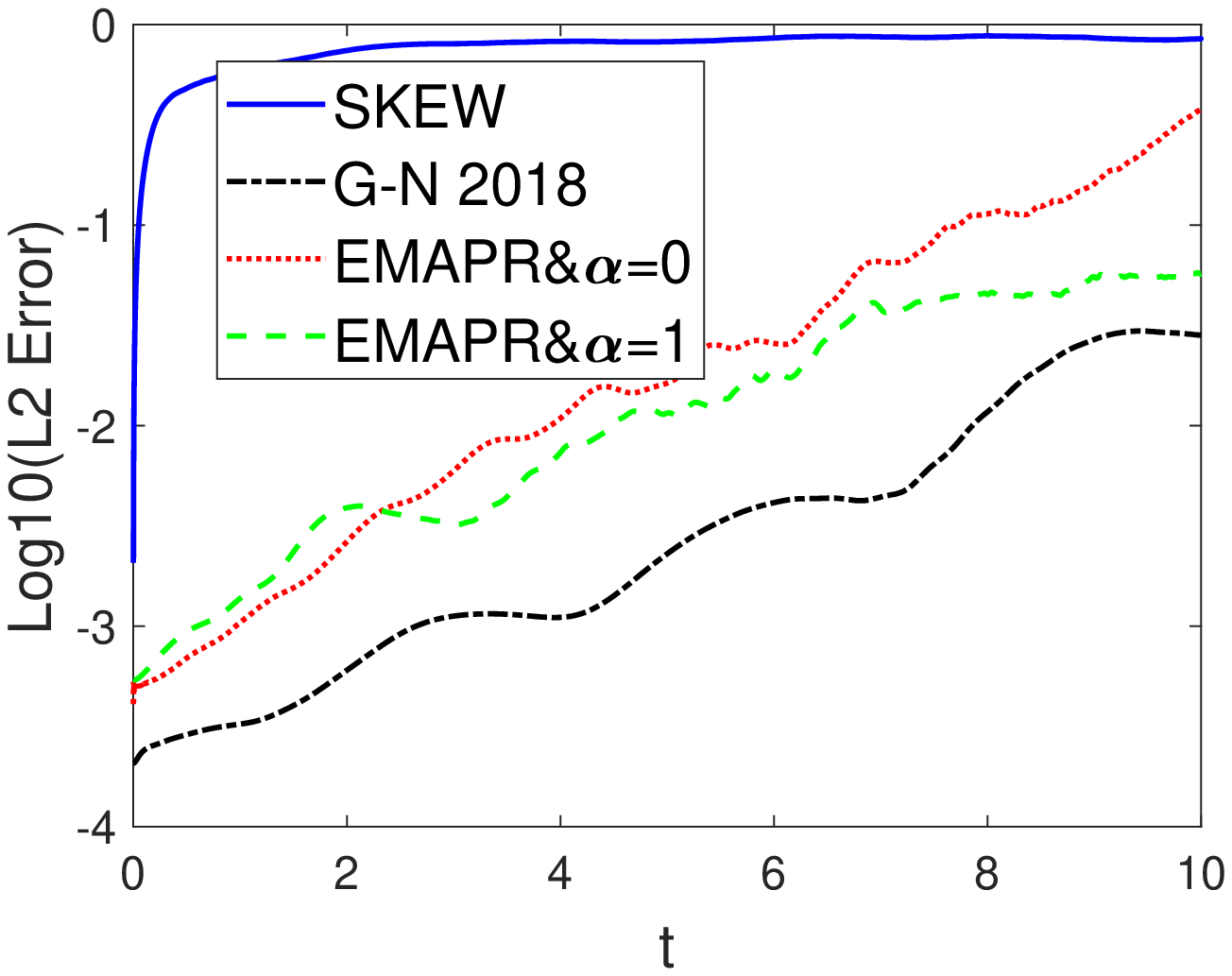}
\includegraphics[width=0.48\textwidth,height=4.2cm]{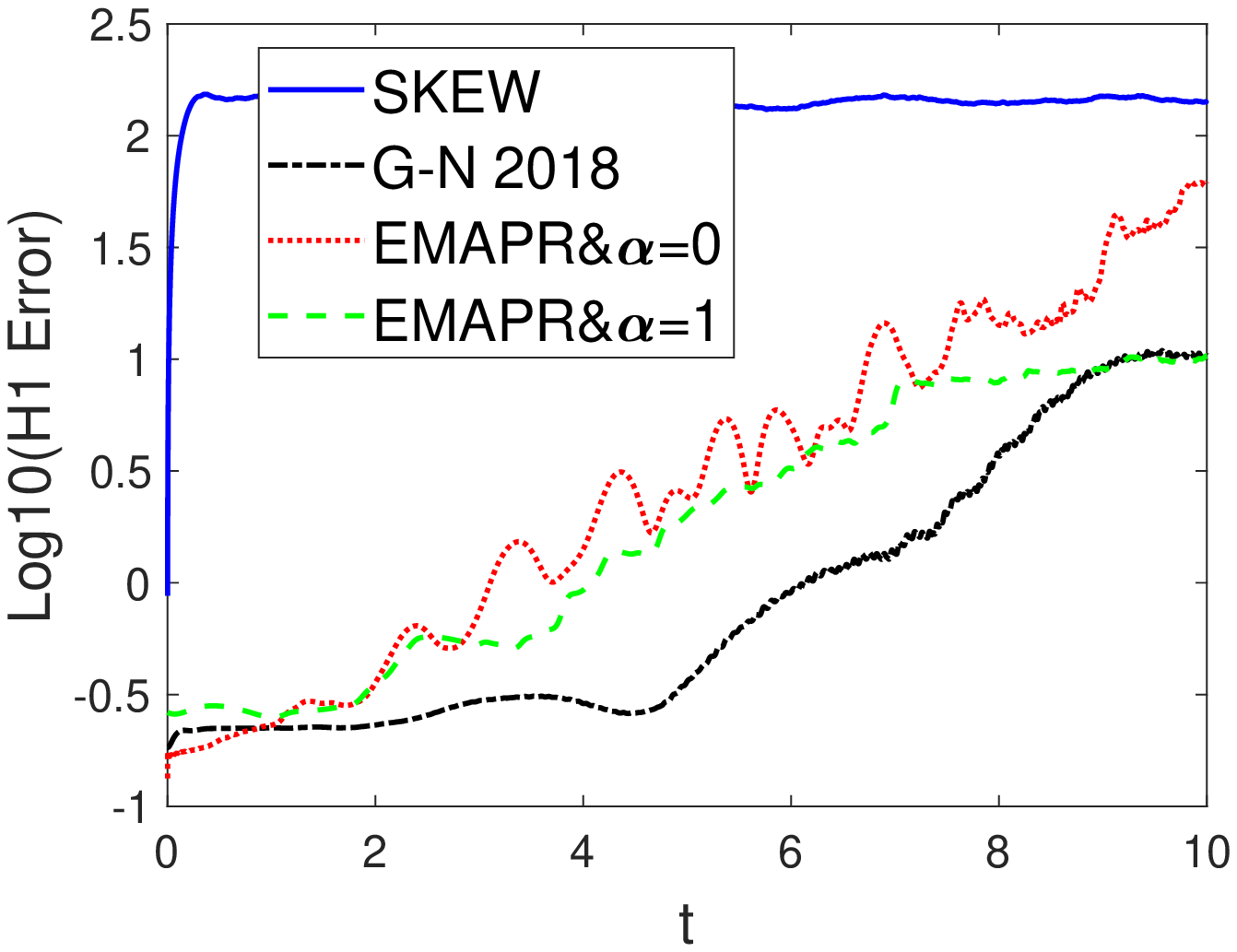}
\caption{Example 3. Plots of L2 errors, H1 errors by the Bernardi-Raugel element or Guzm\'{a}n-Neilan element versus time.}
\label{latticefig1}
\end{figure}
\begin{figure}[htbp]
\centering
\includegraphics[width=0.48\textwidth,height=4.2cm]{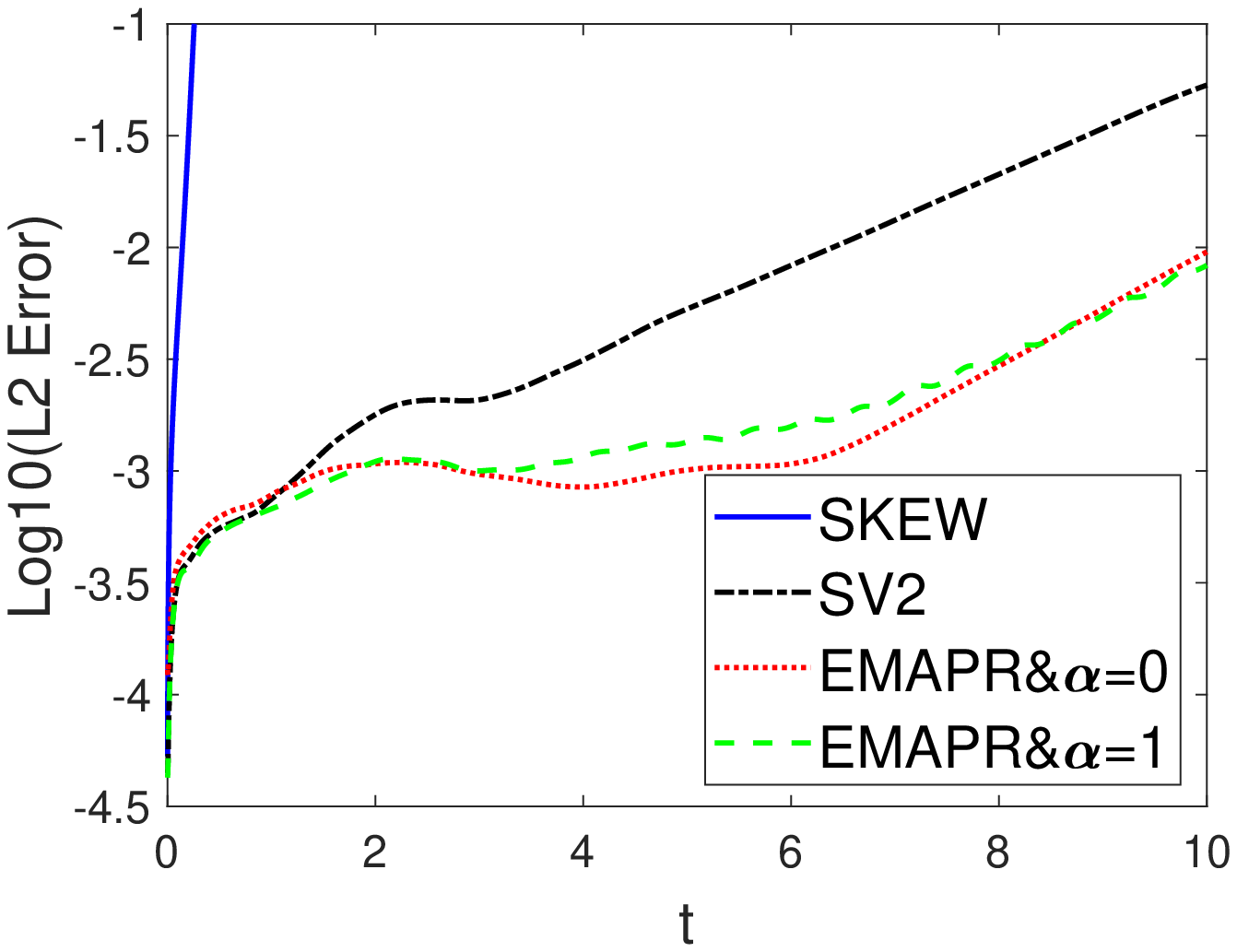}
\includegraphics[width=0.48\textwidth,height=4.2cm]{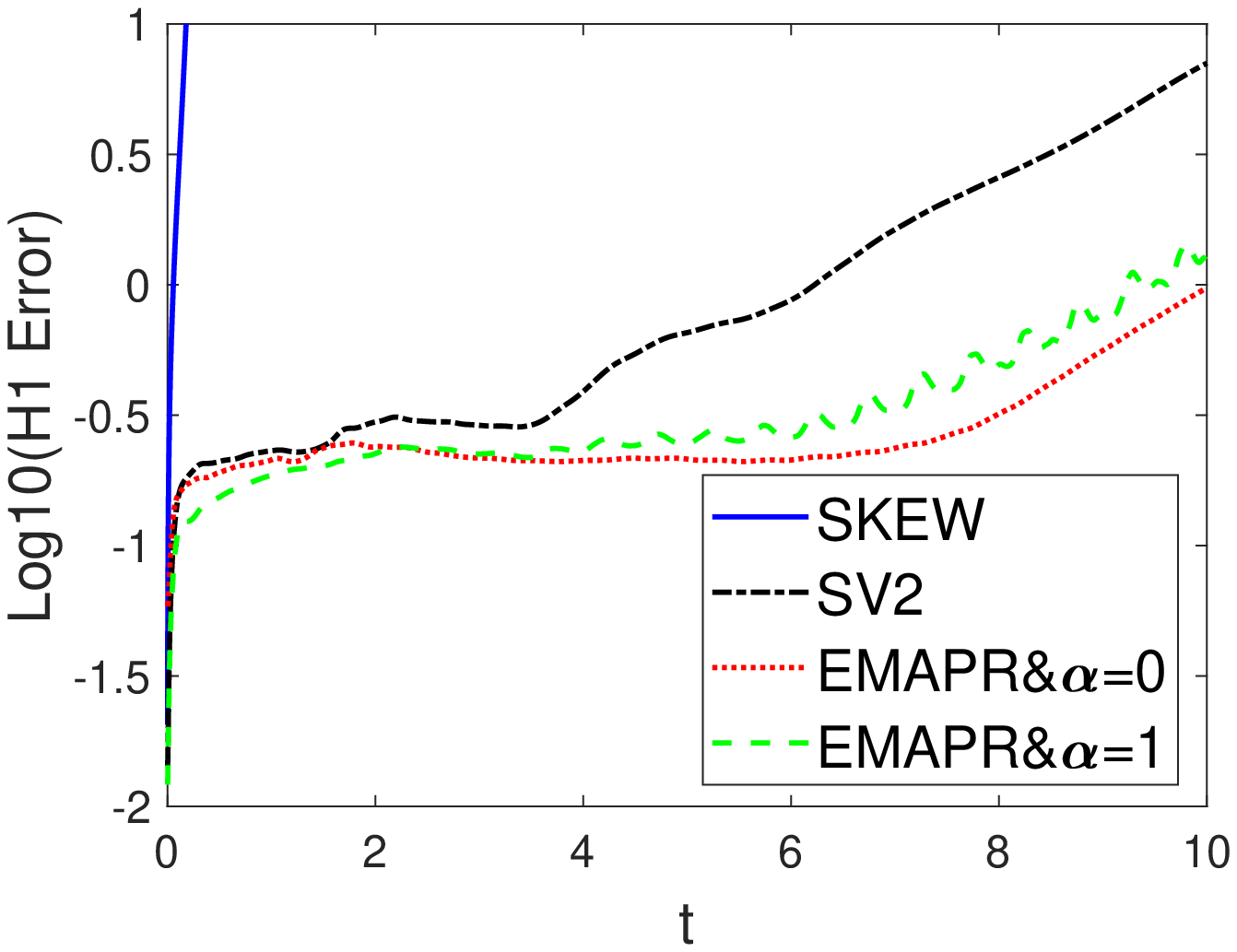}
\caption{Example 3. Plots of L2 errors, H1 errors by $P_{2}^{bubble}/P_{1}^{disc}$ or SV2 versus time.}
\label{latticefig2}
\end{figure}
\bibliographystyle{amsplain}
\bibliography{references}
\end{document}